%% file: main.tex
\documentclass[12pt, a4paper]{report}
\usepackage{polyglossia}

\usepackage{fontspec}
\usepackage[
    backend=biber,   % <— forces Biber
    maxbibnames=50,
    sorting=nty
]{biblatex}

\usepackage{cancel}
\usepackage{float}
\usepackage{scalerel}
\usepackage{amsmath}
\usepackage{amssymb}
\usepackage{relsize} 
\usepackage{tikz}
\usepackage{tikz-cd}
\usepackage{graphicx}
\usepackage{titlesec}
\usepackage{mathrsfs}
\usepackage{extarrows}
\usepackage{etoolbox}
\usepackage{changepage}
\usepackage[none]{hyphenat}
\newcommand{\sqcupscaled}[2][1]{% #1 = scale factor (default 1), #2 = vertical raise 
  \mathbin{%
    \raisebox{#2}{\scalebox{#1}{$\sqcup$}}%
  }%
}
\makeatletter
\usepackage{caption}
\usepackage{dsfont}
\usepackage{nomencl}
\makenomenclature
\usepackage{enumitem}
\DeclareRobustCommand{\LaxSpecThm}{\cite[Theorem~9$'$, p.~363]{Lax}}
\DeclareRobustCommand{\LaxClosedGraphThm}{\cite[Theorem~12, p.~170]{Lax}}
\setdefaultlanguage{english}
\setotherlanguages{english}

\SetLanguageKeys{english}{indentfirst=true}

\usepackage{csquotes}

\DeclareQuoteStyle{english}
  {\quotedblbase}
  {\textquotedblright}
  {\guillemotleft}
  {\guillemotright}

\nocite{*}

\DeclareBibliographyDriver{book}{%
  \printnames{author}%
  \setunit{\labelnamepunct}\newblock
  \printfield{title}%
  \newunit\newblock
  \printlist{publisher}%
  \newunit
  \printlist{location}% (instead of “address” in older bibstyles)
  \newunit\newblock
  \printfield{year}%
  \newunit\newblock
  \printfield{isbn}%    ← first print the ISBN
  \newunit\newblock
  \printfield{note}%    ← then print the note
  \newunit\newblock
  \finentry
}

\usepackage{setspace}
\usepackage{geometry}

\usepackage{graphicx}
\graphicspath{{./images/}}

\usepackage{listings}

\usepackage[
    colorlinks,
    linkcolor={black},
    menucolor={black},
    citecolor={black},
    urlcolor={black}
]{hyperref}

\usepackage[warnings-off={mathtools-colon,mathtools-overbracket}]{unicode-math}

\usepackage{amsmath}
\usepackage{amsthm}
\makeatletter
\renewenvironment{proof}[1][Proof]{%
  \par\noindent\textbf{#1:}\,\,\ignorespaces
}{%
  \\ \hfill\qed\par
}
\makeatother
\usepackage{mathtools}

\newenvironment{abstractpage}
  {\cleardoublepage\vspace*{\fill}}
  {\vfill\cleardoublepage}
\renewenvironment{abstract}[1]
  {\bigskip\selectlanguage{#1}%
   \begin{center}\bfseries\abstractname\end{center}}
  {\par\bigskip}

\usepackage{appendix}

\usepackage{fancyhdr}

\newtheorem{theorem}{Theorem}[chapter]      
\newtheorem{definition}{Definition}[chapter] 
\newtheorem{lemma}{Lemma}[chapter]        
\newtheorem{corollary}{Corollary}[chapter]
\newtheorem{proposition}{Proposition}[chapter]
\newtheorem{example}{Example}[chapter] 

\author{Marin Matei-Luca}

\makeatletter
\newcommand{\restr}[2]{\left.#1\right|_{#2}}

\begin{document}

% Front matter
\cleardoublepage
\let\ps@plain

% Pagina de titlu

\restoregeometry
\newgeometry{
    margin=2.5cm
}

\fancypagestyle{main}{
  \fancyhf{}
  \renewcommand\headrulewidth{0pt}
  \fancyhead[C]{}
  \fancyfoot[C]{\thepage}
}

\addtocounter{page}{1}
\setlength{\parindent}{0pt}
% Rezumatul

\newcommand{\inner}[2]{\left\langle #1, #2 \right\rangle}

% Main matter
\cleardoublepage
\pagestyle{main}
\let\ps@plain\ps@main
\include{0-Title}
\include{0.5-Abstract}
\include{1-Notations}
\include{2-Introduction}
\include{3-Preliminaries}
\include{4-RF}
\include{5-HK}
\include{6-TOP}
\include{7-UHK}

\include{7.1-SP}

\include{8-App}
\include{9-Conclusions}
\include{10-Appendix}

%\bibliographystyle{plain} % Choose an appropriate style
%\bibliography{bibliography}
\printbibliography[heading=bibintoc]

\end{document}

%% file: 0-Title.tex
\begin{titlepage}
\thispagestyle{empty}
\vspace*{2cm} % adjust to push title down if needed

\begin{center}
  % Main Title
  {\scshape\LARGE The Henstock-Kurzweil Functional \\ \vspace{0.5cm} Calculus on Self-Adjoint Operators}
  \vspace{6cm} 
  
  % Submission statement
 {\small
    \textsc{A DISSERTATION SUBMITTED TO THE UNIVERSITY OF MANCHESTER\\
    FOR THE DEGREE OF MASTER OF SCIENCE\\
    IN THE FACULTY OF SCIENCE AND ENGINEERING}
  }

  \vspace{3cm}

  % Year
  \vfill
  {\normalsize 2025}\\
  % Candidate and department
  {\small\bfseries Marin Matei\hspace{+0.01cm}-Luca\\
  Department of Mathematics}

\end{center}
\end{titlepage}

%% file: 0.5-Abstract.tex
\setcounter{page}{1}
\begin{abstractpage}
\begin{abstract}{english}
This dissertation focuses on developing a new construction of a functional calculus using Henstock-Kurzweil integration methods. The assignment of a functional calculus will be applied to self-adjoint operators. We will address both the bounded and unbounded cases, examine the advantage of the underlying function space compared to larger spaces, and explore one application of this functional calculus.
\end{abstract}
\end{abstractpage}
\tableofcontents
\vfill % pushes the text to the bottom of the page
\listoffigures

%% file: 1-Notations.tex
\chapter*{Notations and Assumptions}
\section*{Notations}
\vspace{-1cm}
\allowdisplaybreaks
\begin{adjustwidth}{-1.2cm}{0cm}
\begin{align*}{\hspace{-2cm}}
  \mathbb{F}
  &:= \text{Field, either }\mathbb{R}\text{ or }\mathbb{C}.\\
  \widehat{P}_{\gamma}
  &:= \text{Gauge-fine, tagged partition associated to a compact subset of } \mathbb{R}.\\
   V^*
  &:= \text{The dual space of a given vector space } V.\\
  V^{\perp}
  &:= \text{The orthogonal complement with respect to an inner product vector space}.\\
  \overline{Y}^d
  &:= \text{The closure of Y in the topology induced by } d \text{, a metric or a norm.}\\
  B(x,r)
  &:=\{z \in M \mid \|z-x\|_M < r\} \text{ for some metric apace } (M,d) \text{ with its induced norm}
 . \\
  \partial B(x,r)
  &:=\{z \in M \mid \|z - x\|_M =  r\} \text{ for some metric apace } (M,d) \text{ with its induced norm}
 . \\
  E 
  &:= \text{A projection-valued measure defined on some measurable space.}\\
  Bor(X)
  &:= \text{The Borel }\sigma\text{-algebra of some set } X.\\
  \mathrm{Reg}_{\mathbb{F}}([a,b])
  &:= \text{The set of $\mathbb{F}$-valued regulated functions on some set } X.\\
  \mathrm{C}_{\mathbb{F}}(X)
  &:= \text{The set of continuous $\mathbb{F}$-valued functions on some set } X.\\
  \mathrm{C^k}(I;X)
  &:= \text{The set of $X$-valued k-continuously differentiable functions on some set } I.
  \\
  \mathrm{Borel}_{\mathbb{F}}(X)
   &:= \text{The set of $\mathbb{F}$-valued, $(\mathrm{Bor}(X),\mathrm{Bor}(\mathbb{F}))$-measurable functions}.\\
  \mathrm{Borel}_{b,\mathbb{F}}(X)
  &:= \text{The set of $\mathbb{F}$-valued, bounded, $(\mathrm{Bor}(X),\mathrm{Bor}(\mathbb{F}))$-measurable functions}.\\
  (\mathrm{L^p}(X,\Sigma_X, \lambda),\|\cdot\|_{\mathrm{L^p}},\mathbb{F})
  &:= \text{The } p\text{-th Lebesgue space of some measure space } (X,\Sigma_X,\mu). \\
  (\mathrm{L^p}((X,\Sigma_X, \lambda);V))
  &:= \text{The } p\text{-th Lebesgue space of some measure space } (X,\Sigma_X,\mu). \\
   & \hspace{0.8cm}\text{and a Banach space }(V, \|\cdot\|_V, \mathbb{F}).\\
  \mathrm{Cont}_{\mathbb{F}}(f, X)
  &:=\text{The set of continuous points of an } \mathbb{F}\text{-valued function on a set } X.\\
  \mathrm{Disc}_{\mathbb{F}}(f, X)
  &:=\text{The set of discontinuities of an } \mathbb{F}\text{-valued function on a set } X.\\
  \mathscr{L}(H)
  &:=\text{The set of linear operators on a Hilbert space } H.\\
  \mathrm{B(H)}
  &:= \text{The set of bounded, linear operators on a Hilbert space } H.\\
  Proj_{\perp}(H)
  &:=\text{The set of orthogonal projections associated to a Hilbert space } H.\\
  \sigma(A)
  &:= \text{The spectrum of some operator } A.\\
  \sigma_p(A)
  &:= \text{The point spectrum of some operator } A.\\
  \rho(A)
  &:= \text{The resolvent set of some operator } A.\\
  \left(A, \mathrm{D}(A)\right)
  &:= \text{The identification between an unbounded operator } A \text{ and its domain } \mathrm{D}(A).\\
  (X,\Sigma_X)
  &:= \text{A measurable space for some set } X \text{ and a } \sigma\text{-algebra } \Sigma_X.\\
  supp(\mu)
  &:= \text{The support of a measure on some measure space } (X,\Sigma_X, \mu).\\
   |\mu|(X)
  &:=\text{The total variation measure on some measure space } (X,\Sigma_X, \mu).\\
  (V,\langle\cdot,\cdot\rangle_V,\mathbb{F})
  &:= \text{An inner-product $\mathbb{F}$-vector space}.\\
  (V,\|\cdot\|_V,\mathbb{F})
  &:= \text{A normed $\mathbb{F}$-vector space or a unital Banach algebra}.\\
  (H, \mathbb{F})
  &:= \text{A Hilbert space over } \mathbb{F}.\\
  d(x,Y)
  &:= \inf\{\,d(x,y)\mid y\in Y\}, \text{for some metric space } (M,d), \text{ where} \;x\in M,\;Y\subseteq M.\\
  d_{X \times Y}((x_1,y_1),(x_2,y_2))
  &:=\max\{d_X(x_1,x_2),d_Y(y_1,y_2)\} \text{ for some metric spaces } (X,d_x), (Y,d_Y).\\
  \|\cdot\|_{\infty}&:= \text{The supremum norm on some function space}.\\
  \|\cdot\|_{\mathcal{G}}
  &:= \text{The Graph norm on some product space } V \times W.\\
  \mathrm{Ran}(f)
  &:= \text{Range of a function } f.\\
  ess\,ran_A(f)
  &:= \text{The essential range of a function } f \text{associated to some operator }A.\\
  \Re(f), \Im(f)
  &:= \text{The real and imaginary part of a function } f.\\
  f(x^+)\,,\, f(x^-)
  &:= \text{The right and left limits of } f(x) \text{ in some interval } [a,b].\\
  H_c 
  &:= \text{The Heaviside function with impulse at } c.\\
  \Phi^{C}_A, \Phi^{HK}_A, \Phi^{B}_A 
  &:= \text{The continuous, Henstock Kurzweil and Borel Functional Calculus of} \\
  & \hspace{0.8cm}\text{some operator } A.\\
\end{align*}
\end{adjustwidth}

\section*{Assumptions}
\begin{itemize}
    \item Throughout the whole monograph, the dual space of a Hilbert space will have the functionals taken to be bounded.
    \item For brevity, we will refer to bounded, linear operators as operators, unless we mention a change.
    \item We will assume intermediate to advanced knowledge of the following subjects:
    
    Real Analysis, Topology, Measure Theory, Metric Spaces, and Functional Analysis.
    \\
    \item Given a measure space $(X,\Sigma_X,\mu)$, when we mention the total variation of $\mu$, we will always mean:
    $$|\mu|(X) : = \sup\left\{\sum\limits_{i=1}^n|\mu(A_i)|_{\mathbb{F}} \mid  \left\{A_i\right\}_{i=\overline{1,n}}\subseteq \Sigma_X \text{ \,a partition of }  X,  n \in \mathbb{N}\right\}$$
    \item Given two normed vector spaces $(V, \|\cdot\|_V, \mathbb{F}),\, (W, \|\cdot\|_W, \mathbb{F})$, the graph norm will be defined as follows:
    $$\|(v,w)\|_{\mathcal{G}} : = \|v\|_V + \|w\|_W \quad \text{ for all }  v \in V,\, w\in W$$
    \item The step functions introduced later in Chapters 2 and 3 will have mutually disjoint sets associated with each indicator function.
\end{itemize}

%% file: 2-Introduction.tex
\chapter{Introduction}
Constructing a functional calculus for a self-adjoint operator $T$ is often realised by abstractly defining a spectral integral in relation to a chosen function space. Examples of functional calculi, ordered in ascending order with respect to the size of the corresponding function space, are: the polynomial, holomorphic, continuous, and the Borel functional calculus. 

\bigskip
When $\sigma(T)$ is compact, the first three calculi allow operators to be described as operator limits of sequences of other, usually better-behaved, operators. In this dissertation, we draw inspiration from Ralph Henstock and Jaroslav Kurzweil's approach of generalising the Riemann integral using gauge-fine, tagged partitions, in order to provide a new construction for a functional calculus. With this, we manage somehow accidentally to find out another function space sitting right between the continuous and the Borel functional calculus. Additionally, (if $\sigma(T)$ is compact) the property of having each operator be seen as an operator limit of a convergent sequence of other operators is preserved in this space as well. This is covered in \autoref{ch:henstockkurzweilbounded}.

\bigskip
The Spectral Mapping Theorem is an important theorem relating the spectrum of a symbol $f(T)$ and $\sigma(f(T))$ for some specified function $f$. When $\sigma(T)$ is compact and $f$ is continuous, we get an equality between those two sets:
$$\sigma(f(T)) = f(\sigma(T))$$
For the Borel functional calculus, however, one can only have for certainty the following set inclusion:
$$\sigma(f(T)) \subseteq \overline{f(\sigma(T))}$$
In \autoref{ch: spectralmappingtheorem}, we establish an explicit formula for $\sigma(f(T))$, regardless of whether $\sigma(T)$ is compact or closed. If $\sigma(T)$ is closed, in \autoref{ch: unboundedhk}, we apply a truncating argument, which extends the functional calculus to unbounded self-adjoint operators.

\newpage
These remarks, in and of themselves, suggest that the underlying function space of our functional calculus deserves further analysis: first, as an independent mathematical object (detailed in \autoref{ch: regulatedfunctions}); and second, as a function space in contrast with other spaces (discussed in \autoref{ch: advantageofhk}).

\bigskip
Having the theoretical framework established, in \autoref{ch: applications} we look at an application of our functional calculus in the Theory of Abstract Differential Equations. More specifically, we will use our calculus to represent solutions in a spectral manner and refine one of Amnon Pazy's theorems for our chosen example.

%% file: 3-Preliminaries.tex
\chapter{Preliminaries}
In this chapter, we will establish the underlying concepts and results that we will use in constructing our functional calculus.
\begin{definition}[Jump Discontinuity]
Given a function $f:\mathbb{R}\to\mathbb{F}$, we say that $f$ has a \textbf{jump discontinuity at $y$} if $f(y_{-}),f(y_{+}) \in \mathbb{F}$, but $f(y_{-}) \neq f(y_{+}).$
\end{definition}
\begin{definition}[Removable Discontinuity]
Given a function $f:\mathbb{R}\to\mathbb{F}$, we say that $f$ has a \textbf{removable discontinuity at $y$} if $f(y_{-}) = f(y_{+}) = L$  in $\mathbb{F}$, but $f(y) \neq L$.
\end{definition}

\begin{definition}[Essential Discontinuity]
Given a function $f:\mathbb{R}\to\mathbb{F}$, we say that $f$ has an \textbf{essential discontinuity at $y$} if either $f(y_{-})$ or $f(y_{+})$ does not exist in $\mathbb{F}$.
\end{definition}
\begin{definition}[Gauge function]
Let $X$ be any set.  A \textbf{gauge function} on $X$ is a map
\[
  \gamma:X\to(0,\infty).
\]
\end{definition}
\begin{definition}[Exhaustive $K$-division] Given a compact set $K \subseteq \mathbb{R}$, an \textbf{exhaustive $K$-division} is a finite, strictly-ordered set $\{p_i\}_{i=\overline{0,n}} \subseteq \mathbb{R}$ such that $p_0 < \min(K)$ and $\max(K) > p_n$.
\begin{definition}[$K$-cell]
Given a compact set $K \subseteq \mathbb{R}$ and an exhaustive $K$-division $\{p_i\}_{i=\overline{0,n}}$, a \textbf{$K$-cell} is any set of the form:
$$K \cap [p_{i-1}, p_i] \qquad K\cap [p_{i-1},p_i) \qquad K\cap (p_{i-1},p_i] \qquad K \cap (p_{i-1},p_i)$$
We will denote these $K$-cells with $C_i$. 
\end{definition}
\end{definition}
\begin{definition}[Classical Partition]
A \textbf{partition} $P$ of a compact set $K \subseteq \mathbb{R}$ is a (finite) collection of $K$-chains $\{C_i\}_{i=\overline{1,n}}$ covering $K$.
$$\bigcup_{i=1}^nC_i  = K$$
\end{definition}
\begin{definition}[Tagged Partition]
A \textbf{tagged partition} $\widehat{P}$ of a compact set $K \subseteq \mathbb{R}$ is a partition that has an additional collection of points $\{t_i\}_{i=\overline{1,n}}$ each sitting in precisely one $K$-cell: 
$$t_i \in C_i \quad\text { for } i = \overline{1,n}$$
\end{definition}
We will identify a tagged partition as: $\widehat{P}:= \left\{\left(t_i,C_i\right)]\mid i=\overline{1,n}\right\}$.
\begin{definition}[Gauge-Fine, Tagged Partition]
A \textbf{gauge-fine, tagged partition} $\widehat{P}_{\gamma}$ on a compact set $K \subseteq \mathbb{R}$ is a tagged partition, with the additional constraint that each $K$-chain is controlled by the gauge function evaluated at the corresponding tag.

$$\text{For } i=\overline{1,n}: C_i \subset (t_i - \gamma(t_i), t_i + \gamma(t_i))\cap K$$
\end{definition}
We will identify a gauge-fine tagged partition as:
$\widehat{P}_{\gamma}:=\left\{\left(t_i, C_i, \gamma)\right)\mid i=\overline{1,n}\right\}$.

\begin{theorem}[Separable Metric Space implies at most countable $\epsilon$-separated Subsets]
    Given a separable metric space $(X,d)$ and $\epsilon > 0$, any $\epsilon$-separated subset of $X$ is at most countable.
\end{theorem}
\begin{proof}
We suppose that $S_X$ is a dense subset of $X$. By fixing some $\epsilon > 0$, we assume that there exists some uncountable  $\epsilon$-separated subset $U\subseteq X$. For each $u \in U$, we consider the open ball $B_u : = B\left(u,\frac{\epsilon}{2}\right)$. The collection of all of the open balls $\{B\left(u,\frac{\epsilon}{2}\right)\}_{u \in U}$ is disjoint. 
Indeed, for any distinct $u_1,u_2 \in U$, $x \in B_{u_1}\cap B_{u_2}$ implies:
$$d(u_1, u_2) \leq d(u_1,x) + d(x,u_2) < \frac{\epsilon}{2}+ \frac{\epsilon}{2} = \epsilon$$
This contradicts the fact that $U$ is $\epsilon$-separated.
As $S$ is dense, each non-empty ball $B_u$ contains some point $s_u \in S_X$.

We now consider the following map:
$$\theta: U \longrightarrow S_X, \qquad u \longmapsto s_u$$
This map is injective. For all distinct $u_1,u_2 \in U$, we have $B_{u_1} \cap \, B_{u_2} = \emptyset$, forcing

$s_{u_1} \neq s_{u_2}.$ This shows that  $|U|\leq |S_X|$. Since $U$ is uncountable, so will $S_X$ be. Contradiction.

This proves our claim.
\end{proof}
\begin{theorem}[The  Image of a Linear Isometry between two Banach Spaces is a Banach Space]
Given two Banach spaces $(X, \|\cdot\|_X,\mathbb{F})$, $(Y,\|\cdot\|_Y ,\mathbb{F})$ and a linear isometry 

$T : X \to Y$, we get that $(T(X), \|\cdot\|_{T(X)}, \mathbb{F})$ is a Banach space, where $\|\cdot\|_{T(X)}$ is the restriction of $\|\cdot\|_Y$ to $T(X)$.
\end{theorem}
\begin{proof}
    We need to prove two facts:
\begin{enumerate}[label=\arabic*)]
    \item $T(X)$ is a vector subspace of $Y$
    \item $(T(X), \|\cdot\|_{T(X)}, \mathbb{F})$  is complete
\end{enumerate}
As $1)$ is an immediate consequence of the linearity of $T$, we will elaborate on the second fact.

For $2)$, we consider an arbitrary sequence $\{y_n\}_{n \in \mathbb{N}} \subseteq T(X): y_n \xrightarrow[\;n\to\,\infty\;]{\|\cdot\|_{Y}} y \in Y$.

Since for each $n \in \mathbb{N}: y_n = T(x_n)$ for some unique $x_n$, we get, by the isometry property, that:
$$\|x_n - x_m\|_X = \|T(x_n) - T(x_m)\|_Y=\|y_n - y_m\|_Y \xrightarrow[\,n,\,m\,\to\, \infty]{|\cdot|_{\mathbb{F}}} 0$$
This proves that $\{x_n\}_{n \in \mathbb{N}}$ is Cauchy. By the completeness of $(X, \|\cdot\|_X,\mathbb{F})$, we have:
$$x_n\xrightarrow[\,n\,\to \,\infty]{\|\cdot\|_{\mathbb{X}}} x \,\,\text{ with } x \in X$$
But then:
$$ T(x_n) \xrightarrow[\,n\,\to \,\infty]{\|\cdot\|_{\mathbb{Y}}} T(x)$$
As a consequence of the uniqueness of limits in Banach spaces, we get:
$$ y_n  \xrightarrow[\,n\,\to \,\infty]{\|\cdot\|_{\mathbb{Y}}} y \,\,\text{ with } y =T(x) \in T(X)$$
This proves the second fact, finishing our claim. 
\end{proof}

\begin{theorem}[The Cauchy-Schwarz Inequality]
    Given an inner product vector space $(V,\langle\cdot,\cdot\rangle_V,\mathbb{F})$, the following inequality holds:
    $$\text{For all } u,v \in V: |\inner{u}{v}_V|_{\mathbb{F}} \leq \|u\|_V\|v\|_V$$
\end{theorem}

\begin{proof}
If $v = 0$, then both sides will be zero. 

We suppose $v\neq 0$. Then, we may decompose any $u$ in $V$ as follows: 
$$u = u_1 + u_2 \text{ where } u_1:= \frac{\inner{u}{v}_V}{\|v\|^2_V}v \,\,\text{  and  }\,\, u_2:= u - \frac{\inner{u}{v}_V}{\|v\|^2_V}v$$
By the fact that $u_1$ and $u_2$ are orthogonal, we get:
$$\|u\|^2_V = \|u_1\|^2_V + \|u_2\|^2_V \geq \|u_1\|^2_V = \frac{\left|\inner{u}{v}_V\right|_{\mathbb{F}}}{\|v\|^2_V}$$
Multiplying through by $\|v\|^2_V$, we get our desired result:
$$|\inner{u}{v}_V|_{\mathbb{F}} \leq \|u\|_V\|v\|_V$$
This finishes our proof.
\end{proof}
\begin{theorem}[A Sufficient Criterion for the Spectrum of a Closed, Unbounded Operator]
Given a Hilbert space $(H,\mathbb{F})$, a closed, unbounded operator $(A,\mathrm{D}(A))$ on $H$ and some scalar $\lambda  \in \mathbb{C}$, if there exists a unit-normed sequence $\{x_n\}_{n \in \mathbb{N}} \subseteq H : (A - \lambda I)x_n \xrightarrow[\,n \, \to \,\infty]{\|\cdot\|_{H}} 0$, then $\lambda \in \sigma(A)$.
\end{theorem}
\begin{proof}

Let $\{x_n\}_{n \in \mathbb{N}} \subseteq H$ be a unit-normed sequence such that:
$$ (A - \lambda I)x_n \xrightarrow[\,n \, \to \,\infty]{\|\cdot\|_{H}} 0$$
We suppose that $\lambda \in  \rho(A)$. Then $A- \lambda I$ is a bijection. In particular, we have a closed, bounded operator:
$$(A-\lambda I)^{-1} : H \longrightarrow H$$
We consider the following sequence: $y_n : = (A- \lambda I)x_n$ for all $n \in \mathbb{N}$.
Then:
$$\text{For all } n \in \mathbb{N}:\|x_n\|_H = \|(A-\lambda I)^{-1}(A - \lambda I)x_n\|_H \leq \|(A- \lambda I)^{-1}\|_{op}\|y_n\|_H  \xrightarrow[\,n \, \to \,\infty]{|\cdot|_{\mathbb{R}}} 0$$
This contradicts the fact that $\{x_n\}_{n \in \mathbb{N}}$ is unit-normed. As a consequence, $\lambda \in \sigma(A)$.

This proves our claim.
\end{proof}
\begin{theorem}[Unbounded Normal Operators are Closed]
Given a Hilbert space $(H, \mathbb{F})$, and a normal, unbounded operator $(A,\mathrm{D}(A))$ on $H$, $A$ is closed.
\end{theorem}
\begin{proof}
For the sake of brevity, we will let $\|\cdot\|_{\mathcal{G}(A)}$ and $\|\cdot\|_{\mathcal{G}(A^{*})}$ be the graph norms associated to $A$ and $A^*$. Let $(A, \mathrm{D}(A))$ be a normal, unbounded operator on $H$. Since $A^*$ is closed, we have that $(D(A^*), \|\cdot\|_{\mathcal{G}(A^*)}, \mathbb{F})$ is a Banach space. By the normality of $A$, we get two facts:
\begin{enumerate}
    \item $\mathrm{D}(A) = \mathrm{D}(A^*)$ 
    \item For all  $x \in \mathrm{D}(A): \|x\|_{\mathcal{G}(A)}: = \|x\|_{\mathrm{D}(A)} + \|Ax\|_{H} = \|x\|_{\mathrm{D}(A^*)} + \|A^*x\|_H = \|x\|_{\mathcal{G}(A^*)}$
\end{enumerate}
This implies that $(D(A^*), \|\cdot\|_{\mathcal{G}(A^*)}, \mathbb{F}) = (D(A), \|\cdot\|_{\mathcal{G}(A)}, \mathbb{F})$. As  $(D(A^*), \|\cdot\|_{\mathcal{G}(A^*)}, \mathbb{F})$ is a Banach space, so will $(D(A), \|\cdot\|_{\mathcal{G}(A)}, \mathbb{F})$ be a Banach space. Equivalently, $A$ is closed.

This proves the closedness of unbounded normal operators.
\end{proof}

\begin{theorem}[A Different Way of seeing Norms induced by Inner Products]
Given an inner product vector space  $(V,\langle\cdot,\cdot\rangle_V,\mathbb{F})$, we have:
$$\text{For all } v \in V: \,\,\|v\|_{H} \,= \hspace{-0.2cm}\sup\limits_{u \in V,\ \|u\|_V = 1} |\inner{v}{u}_V|_{\mathbb{F}}$$
\end{theorem}
\begin{proof}
Let $v \in V$. As $v = 0$ implies that both sides will be zero, we suppose $v \neq 0$.

For $\leq$, we consider $ w = \displaystyle\frac{v}{\,\,\,\,\|v\|_{V}}$.

Then:

$$|\inner{v}{w}_V|_{\mathbb{F}} = \frac{\|v\|^2_V}{\|v\|_V} = \|v\|_V$$

This shows:
$$  \,\,\|v\|_{H} \,\leq\hspace{-0.2cm}\sup\limits_{u \in V,\ \|u\|_V = 1} |\inner{v}{u}_V|_{\mathbb{F}}$$.

For $\geq$, we will make use of the Cauchy-Schwarz inequality, directly leading to the claim.
$$\text{For all } u,v \in V: |\inner{v}{u}_V|_{\mathbb{F}} \leq \|v\|_V\|u\|_V$$
By taking the supremum over all unit-normed values $w \in  H$, we get:
$$ \sup\limits_{u \in V,\ \|u\|_V = 1} |\inner{v}{u}_V|_{\mathbb{F}} \,\leq  \|v\|_{H} $$
This proves our claim.
\end{proof}
\begin{theorem}[Riesz's Representation Theorem]
Given a Hilbert space $(H,\mathbb{F})$, there exists a unique antilinear isometric isomorphism between $H$ and $H^*$ given by:
\begin{equation*}
\begin{aligned}
  \varphi\colon\;&H \;\longrightarrow\; H^*\\
  &v\;\longmapsto\; f_v := \inner{\cdot}{v}_H
\end{aligned}
\end{equation*}
\end{theorem}
\begin{proof}
If $H = \{0\}$, then the statement holds trivially.
We will assume that $H\neq\{0\}$. As any inner product on a vector space is antilinear in the second argument, we get that our map is antilinear. 
We will now move on to prove, in this order, isometry, bijectivity, and uniqueness of the map.

\textbf{Isometry:}

For $v= 0$, we have $f_v = 0_{H^*}$, establishing the isometry condition: 
$$\|f_v\|_{op}=\sup_{u\in H,\ \|u\|_H=1}|f_v(u)| = \|v\|_H = 0 $$
For $v\neq 0$, we get:
$$\|f_v\|_{op}=\sup_{u\in H,\ \|u\|_H=1}|f_v(u)|  =  \sup_{u\in H,\ \|u\|_H=1} |\inner{u}{v}_H|_{\mathbb{F}}$$
By the Cauchy-Schwarz inequality, we have:
$$|\inner{u}{v}_H|_{\mathbb{F}} \leq \|u\|_H\|v\|_H \,\,\,\text{ implies } \,\, \sup_{u\in H,\ \|u\|_H=1}|\inner{u}{v}_H|_{\mathbb{F}} \leq \|v\|_H$$
For the other inequality, taking $u = \frac{v}{\|v\|_H}$ yields:
$$\sup_{u\in H,\ \|u\|_H=1}|\inner{u}{v}_H|_{\mathbb{F}} \geq \left|\inner{\frac{v}{\|v\|_H}}{v}\right|_{\mathbb{F}} = \|v\|_H$$
Hence, from the last two inequalities, we get our isometry:
$$\|f_v\|_{op} =\|v\|_H \,\,\text{ for all } \,\, v \in H$$
\textbf{Bijectivity:} 

For $f_v = 0_{H^*}$, there is only one value $v$ in $H$ such that $\inner{\cdot}{v}_H = 0$, namely $v = 0$.

We suppose $f_v\neq 0_{H^*}$. Since $f_v$ is bounded, hence continuous, and singletons are closed in $\mathbb{F}$, we get that $\mathrm{Ker}(f_v) = f_v^{-1}(\{0\})$ is a closed subspace of $H$. As $f_v\neq 0_{H^*}$ implies $\mathrm{Ker}(f_v) \subsetneq H$, there must be an element in $z \in H$ such that $f_v(z) \neq 0.$ We claim the following:
$$H = \mathrm{Ker}(f_v) \bigoplus \mathrm{Span}(z)$$
By this, we mean that every element $x$ in $H$ can be written as the sum of two other elements, each lying in only one of these two spaces. Additionally, we will also require that $\mathrm{Ker}(f_v) \cap \mathrm{Span}(z) = \{0\}$.

We will start by considering an arbitrary element $x$ in $H$, which we will decompose as such:
$$ x  = (x - \lambda_x z) + \lambda_x z \quad\text{ for some } \quad \lambda_x \in \mathbb{F}$$
We want the first term to belong to $\mathrm{Ker}(f_v)$, hence finding out our scalar becomes a simple algebraic play:
$$x - \lambda_xz \in \mathrm{Ker}(f_v) \text{ if and only if } f_v(x - \lambda_xz) = 0  \text{ if and only if } \lambda_x = \frac{f_v(x)}{f_v(z)}$$
Hence, our decomposition becomes: 
$$ x  = \left(x - \frac{f_v(x)}{f_v(z)}z\right) + \frac{f_v(x)}{f_v(z)}z$$
We may now wish to show that $\mathrm{Ker}(f_v) \cap \mathrm{Span}(z) = \{0\}$. Any element in the intersection must be of the form $\lambda z$ for some $\lambda \in \mathbb{F}$, such that $f_v(\lambda z) = \lambda f_v(z) = 0$. Since $f_v(z) \neq 0$, we have that $\lambda = 0$, implying: $\{0\} \subseteq \mathrm{Ker}(f_v) \cap \mathrm{Span}(z) \subseteq \{0\}$, and hence, proving $\mathrm{Ker}(f_v) \cap \mathrm{Span}(z) = \{0\}$. With some additional work, which we will omit, one can also prove that $\mathrm{Span}(z) = \mathrm{Ker}(f_v)^{\perp}$. Without loss of generality, we will normalize $z$. By fixing an arbitrary $x$ in $H$ and setting $v = \overline{f_v(z)}z$, besides getting the fact that our function $f_v$ is well-defined under $x$, we have:
\begin{align*}
    \inner{x}{v}_H 
    &= \inner{\left(x - \frac{f_v(x)}{f_v(z)}z\right) + \frac{f_v(x)}{f_v(z)}z}{\overline{f_v(z)}z}_H\\
    &= \inner{\frac{f_v(x)}{f_v(z)}z}{\overline{f_v(z)}z}_H\\
    &= f_v(x)
\end{align*}
This proof of the existence of $v$ is constructive in the sense that it explicitly shows who our candidate is.

To finalise bijectivity, it remains to show uniqueness for $v$. We suppose that there exist two values $v_1, v_2$ in $H$ such that $\inner{\cdot}{v_1} = \inner{\cdot}{v_2}$  However, evaluating both sides at $v_1 - v_2$ will directly give us $v_1 = v_2$, proving uniqueness.
\medskip

\textbf{Uniqueness of the Map:}

Suppose there are two such maps $\varphi_1,\varphi_2$. 

Then, for each $v$ in $H$ we have $\varphi_1(v) = \inner{\cdot}{v}_H = \varphi_2(v)$. 

Since they agree as functionals everywhere in $H$, they are indeed the same.

This finishes our proof.
\end{proof}
\bigskip
\bigskip
With the help of Riesz’s Representation Theorem, we obtain arguably the most important bijection for the construction of our new functional calculus.
\begin{corollary}[The Bijection between Sesquilinear Forms and Bounded, Linear Operators]
Given a bounded, sesquilinear form $\eta$ on a Hilbert space $(H,\mathbb{F})$, there exists a unique operator $T$ in $\mathrm{B}(H)$ such that:
$$\text{For all } x,y \in H: \eta(x,y) = \inner{Tx}{y}_H$$
Furthermore:
$$\|T\|_{op} = \|\eta\|_{op}:=\sup\{|\eta(x,y)|_{\mathbb{F}}\mid x,y \in H \text{ and } \|x\|_H = \|y\|_H =1\}$$
\end{corollary}

\begin{proof}
   As any sesquilinear form is antilinear in its second argument, we will consider a homogenised variant. We fix an arbitrary element $x$ in $H$, a sesquilinear form $\eta$, and consider the following bounded, linear map:
\begin{equation*}
\begin{aligned}
  f_x\colon\;&H \;\longrightarrow\; \mathbb{F}\\
  &y\;\longmapsto\; \overline{\eta(x,y)}
\end{aligned}
\end{equation*}
Since $f_x$ lives in $H^*$, we may invoke Riesz's Representation Theorem to obtain a unique element $Tx$ in $H$ such that:
$$ f_x(y) = \inner{y}{Tx} = \overline{\inner{Tx}{y}}$$
This directly yields a direct correspondence with the chosen sesquilinear form $\eta$:
$$\eta(x,y) = \inner{Tx}{y}_H$$
We will now move on to prove, in this order, the linearity, boundedness, and uniqueness of $T$.
\medskip

\textbf{Linearity:}

For all $x_1,x_2,y \in H \text{ and }\alpha_1,\alpha_2 \in \mathbb{F}$, we have:
\begin{align*}
    \inner{T(\alpha_1 x_1 + \alpha_2 x_2)}{y}_H 
    &= \eta(\alpha_1 x_1 + \alpha_2 x_2, y)\\
    &=\alpha_1\eta(x_1,y) + \alpha_2\eta(x_2,y)\\
    &= \alpha_1\inner{Tx_1}{y}_H + \alpha_2\inner{Tx_2}{y}_H\\
    &=\inner{\alpha_1 Tx_1}{y}_H +\inner{\alpha_2Tx_2}{y}_H\\
    &=\inner{\alpha_1 Tx_1 + \alpha_2 Tx_2}{y}_H\\
\end{align*}
This proves linearity.
\medskip

\textbf{Boundedness:}

For all $x \in H$, we have:
$$\|Tx\|^2_H = \inner{Tx}{Tx}_H = \eta(x, Tx) = |\eta(x,Tx)|_{\mathbb{F}} \leq \|\eta\|_{op}\|x\|_H\|Tx\|_H$$
\textbf{Remark:} Here, it is worth mentioning that the sesquilinear form agrees with its modulus, as taking norms yields a positive value.

Dividing by $\|Tx\|_H$ on both sides and taking the supremum over all normalized values of $x$ in $H$, we obtain: 
$$\|T\|_{op} \leq \|\eta\|_{op}$$
This proves boundedness.

Since we are one inequality away from establishing the fact that these two norms coincide, we will seize the opportunity. Starting from the sesquilinear form, we have:
$$|\eta(x,y)|_{\mathbb{F}}\leq |\inner{Tx}{y}_H|_{\mathbb{F}} \leq \|T\|_{op}\|\|x\|_H\|y\|_H$$
By taking the supremum over all normalized values of $x$ and $y$ in $H$, we get the final result:
$$\|\eta\|_{op} \leq \|T\|_{op}$$
This establishes the equality between the two norms. 
\medskip

\textbf{Uniqueness:}

We suppose there are two such operators $T_1,T_2$. 

Then, for all $x,y \in H$ we have $\inner{T_1x}{y}_H = \inner{T_2x}{y}_H$. By letting $y = \left(T_1 - T_2\right)x$, we establish that the two operators coincide.

This finishes both uniqueness and the proof as a whole.
\end{proof}

\begin{theorem}[Closed Graph Theorem \LaxClosedGraphThm]\label{thm:closedgraphtheorem}
Given two Banach spaces $(V_1, \|\cdot\|_{V_1}, \mathbb{F}), (V_2, \|\cdot\|_{V_2}, \mathbb{F})$ and a (not necessarily bounded) linear operator $T : V_1 \longrightarrow V_2$, the following are equivalent:
\begin{enumerate}
    \item T is closed
    \item T is bounded
\end{enumerate}
\end{theorem}
\begin{theorem}[The Spectral Theorem for Self-Adjoint Operators \LaxSpecThm]
Given a Hilbert space $(H,\mathbb{F})$ and a self-adjoint operator $T \in \mathrm{B}(H)$, there exists a unique projection-valued measure $E$ on $(\sigma(T), Bor(\sigma(T))$ such that:
$$T = \int\limits_{\sigma(T)}\lambda \,dE(\lambda)$$

In this case, $E$ is called the resolution of the identity for $T$.
\end{theorem}
\begin{definition}[The Essential Range]
    Given a Hilbert space $(H,\mathbb{F})$, a self-adjoint operator $A \in \mathrm{B}(H)$, its (unique) resolution at the identity $E$, the \textbf{essential range} of a Borel-measurable function $f \in \mathrm{Borel}_{\mathbb{F}}(\sigma(A))$ is defined as follows:
    $$ess\,ran_A(f):= \left\{\lambda \in \mathbb{C} \mid \text{for all } \epsilon > 0: E(f^{-1}(B(\lambda,\epsilon))  \neq 
 0\right\}$$
\end{definition}
\begin{theorem}[Complex Measure of Bounded Variation]
Given a Hilbert space $(H,\mathbb{F})$, a self-adjoint operator $A \in \mathrm{B}(H)$ and its (unique) resolution at the identity $E$, we have:
$$\text{For all } x,y \in H: \mu_{x,y}(\cdot):=\inner{E(\cdot)x}{y}_H \text{ is a complex measure of bounded total variation}.$$
\end{theorem}
\begin{proof}
    We will begin by fixing $x,y \in H$.
\medskip

\textbf{Complex Measure:}
\begin{enumerate}
    \item $\mu_{x,y}(\emptyset) = \inner{0}{y}_H = 0$
    \item For all countable collection of disjoint Borel sets $\{B_i\}_{i \in \mathbb{N}}\subseteq Bor(\sigma(A))$ we have:
    $$\mu_{x,y}\left(\bigsqcup\limits_{i=1}^{\infty}B_i\right) = \inner{E\left(\bigsqcup\limits_{i=1}^{\infty}B_i\right)x}{y}_H = \sum\limits_{i=1}^{\infty}\inner{E(B_i)x}{y}_H = \sum\limits_{i=1}^\infty\mu_{x,y}(B_i)$$
\end{enumerate}
The second equality is easily justified by using the properties of a projection-valued measure and the Cauchy-Schwarz inequality:
$$\hspace{-0.48cm}\lim\limits_{n \to \infty}\left|\inner{\sum\limits_{i=1}^nE(B_i)x - E\left(\bigsqcup\limits_{i=1}^{\infty}B_i\right)x}{y}\right|_{\mathbb{F}}  \leq \lim\limits_{n \to \infty}\left\|\left(\sum\limits_{i=1}^nE(B_i) - E\left(\bigsqcup\limits_{i=1}^{\infty}B_i\right)\right)x\right\|_H\|y\|_H = 0 $$ 

This proves: $$\lim\limits_{n \to \infty}\sum\limits_{i=1}^n\inner{E(B_i)x}{y}_H = \inner{E\left(\bigsqcup\limits_{i=1}^{\infty}B_i\right)x}{y}_H$$
\medskip

\textbf{Bounded Variation:}

We fix an arbitrary disjoint collection of Borel sets partitioning $\sigma(A):\{A_i\}_{i=\overline{1,n}}$.

Some direct inequalities yield:
\begin{align*}
    \sum\limits_{i=1}^n\left|\mu_{x,y}(A_i)\right|_{\mathbb{F}} =  \sum\limits_{i=1}^n\left|\inner{E(A_i)x}{y}_H\right|_{\mathbb{F}} \leq \sum\limits_{i=1}^n \left\|E(A_i)x\right\|_H\|y\|_H \leq ||x||_H\|y\|_H
\end{align*}
Since the partition was arbitrarily chosen, we conclude:
$$|\mu_{x,y}|\left(\sigma(A)\right) \leq \|x\|_H \|y\|_H < \infty \text{ for all } x,y \in H$$
This proves boundedness, finishing off the proof.
\medskip

 We will now formalise in complete generality what one means by a functional calculus associated to an operator $T$. 
\end{proof}

\begin{definition}[Functional Calculus]
Let $(\mathscr{F}^*(X,\mathbb{F}), \mathbb{F})$ be a $*$-algebra of $\mathbb{F}$-valued functions on some set $X$ where the involution is given by complex conjugation.
Given the unital Banach algebra $(\mathrm{B}(H), \|\cdot\|_{op}, \mathbb{F})$, and $T$ in $\mathrm{B}(H)$, we say $T$ admits a \textbf{functional calculus} if there exists a $*$-algebra homomorphism between $\mathscr{F}^*(X,\mathbb{F})$ and $\mathrm{B}(H)$. Concretely, there exists a map:

\[
  \Phi_T\colon
  \mathscr{F}^*(X,\mathbb F)
  \;\longrightarrow\;
  \mathrm{B(H)},
  \quad
  f\;\longmapsto\;f(T)
\]
such that for all $f_1,f_2,g \in \mathscr{F}^*(X,\mathbb{F})$ and $\alpha_1,\alpha_2 \in \mathbb{F}$:
\begin{enumerate}[label=\arabic*)]
\item $\Phi_T(\alpha_1 f_1 + \alpha_2f_2) = \alpha_1\Phi_T(f_1) + \alpha_2 \Phi_T(f_2)$
\item $\Phi_T(f_1f_2) = \Phi_T(f_1)\Phi_T(f_2)$
\item $\Phi_T(\overline{g}) = (\Phi_T(g))^*$
\item $\Phi_T(1_{ \mathscr{F}^*(X,\mathbb F)}) = 1_{\mathrm{B}(H)}$
\end{enumerate}
\end{definition}
\medskip

One usually endows  $\mathscr{F}^*(X,\mathbb{F})$ with a norm, under which the $*$-algebra becomes a Banach space as well. This is because one can then extend the functional calculus to a bigger function space, in which the initial space is dense. This is precisely how one extends the polynomial functional calculus to the continuous functional calculus. In this regard, the construction will be analogous. We first identify a dense subspace of $\mathrm{Reg}_{\mathbb{F}}([a,b])$, define our functional calculus on this smaller space, and then extend it to the larger one.

%% file: 4-RF.tex
\chapter{Regulated Functions} \label{ch: regulatedfunctions}
In this chapter, we will set up the space in which our functional calculus will be defined. By setting up, we mean characterising the functions by an equivalent definition. Through the established equivalence and properties, we will have an easier time constructing our functional calculus. We may begin with one of the definitions of a regulated function.
\begin{definition}[Regulated Functions]
Given a compact set $K \subseteq \mathbb{R}$, a function is \textbf{regulated} if all left and right limits exist on $K$.
\end{definition}

We now wish to have not one, but three equivalent ways of recognising a regulated function. 
\begin{theorem}[Equivalent Definitions for Regulated Functions]
The following statements are equivalent:
\begin{enumerate}
    \item $f \in \mathrm{Reg}_{\mathbb{F}}(K)$
    \item There exists a sequence of step functions $\{s_n\}_{n \in \mathbb{N}}$ on $K$ such that  $s_n
      \xrightarrow[\;n\to\infty\;]{\|\cdot\|_\infty}
      f$
    \item The set of all discontinuities of $f$ is at most countable, consisting only of jump and removable ones.
\end{enumerate}
\end{theorem}
\begin{proof}
We will prove the equivalence by 
considering the following implications:
\begin{figure}[h]
\centering
\begin{tikzpicture}[%
    every node/.style={font=\sffamily,inner sep=1pt},
    twohead/.style={<->,thick,gray!70},
    single/.style={->,thick,gray!70}
  ]
  % Nodes
  \node (n1) at (0,1.5) {1};
  \node (n2) at (-1,0)  {2};
  \node (n3) at ( 1,0)  {3};

  % Arrows
  \draw[twohead] (n1) -- (n2);   % 1 <-> 2
  \draw[single]  (n2) -- (n3);   % 2 -> 3
  \draw[single]  (n3) -- (n1);   % 3 -> 1
\end{tikzpicture}
\end{figure}
\newpage
For $1 \iff 2$, we first consider a regulated function $f\in \mathrm{Reg}_{\mathbb{F}}(K)$.

For each $x \in K$, we fix $\epsilon >0$ and choose some threshold $\delta_x > 0$ such that:
$$|f(t) - f(x^+))|_{\mathbb{F}}< \epsilon \text{ as } t \in (x,x+\delta_x) \cap K$$
$$|f(t) - f(x^-)|_{\mathbb{F}}< \epsilon \text{ as } t \in (x-\delta_x,x) \cap K$$

Now, we look at the subspace topology of $K$.
$$\text{Let } I_x : = (x - \delta_x, x+\delta_x) \cap K $$
By Heine-Borel, this open cover admits a finite open subcover $\{I_{x_i}\}_{i =\overline{1,n}}$.  

We will consider $$E =\left\{x_i - \delta_{x_i}, x_i, x_i + \delta_{x_i} \mid i = \overline{1,n}\right\}$$ By taking the complement, we obtain $\mathbb{R}\setminus E = \{(a_i,b_i)\}_{i=\overline{1,m}}$, giving rise to the following partition of $K$:
$$ \text{For all }  i = \overline{1,m}: C_i = K \cap (a_i,b_i) \,\,\, \text{and} \,\,\,   i =\overline{m+1, m+n}: C_i = \left\{p_i\right\}  \text{ for } p_i \in K \cap E$$
We are now ready to construct our step function: 
$$ 
s(t) = 
\begin{cases}
f(p_i), &\text{for }\,\,t =p_i\\
f(x^-_i),&\text{for }\,\, t \in C_i \subseteq K \cap (x_i - \delta_{x_i}, x_i) \\
f(x^+_i),&\text{for }\,\, t \in C_i\subseteq K \cap (x_i, x_i + \delta_{x_i})
\end{cases}
$$

By exhausting all possibilities:
\begin{enumerate}
    \item $|f(p_i) - s(p_i)|_{\mathbb{F}}$ = 0
    \item $|f(t) - s(t)|_{\mathbb{F}} = |f(t) - f(x^+_i)|_{\mathbb{F}} < \epsilon$ as $t \in (x_i - \delta_{x_i}, x_i)$
    \item $|f(t) - s(t)|_{\mathbb{F}} = |f(x) - f(x^-_i)|_{\mathbb{F}} < \epsilon$ as $t \in (x_i , x_i + \delta_{x_i})$
\end{enumerate}
we get $\| f - s\|_\infty < \epsilon \text{ for all } \epsilon > 0.$ 

This proves the forward direction.
\medskip

For the reverse direction, we suppose that our function $f$ is uniformly approximated by a sequence of step functions $\{s_n\}_{n \in \mathbb{N}}$: 
$$\text{For all } \epsilon > 0 \text{ there exists } N \in \mathbb{N} \text{ such that for all } n\geq N: \|f -s_n\|_{\infty}<\frac{\epsilon}{2}$$
As $s_N$ is a step function, it can be made constant on each disjoint component of the partition $\{K \cap (a_i,b_i)\mid i =\overline{1,n}\}$. 
\newpage 
By letting $\mathrm{Disc}_{\mathbb{F}}(s_N, K)$ be the set of discontinuities of $s_n$, we will consider two cases:
\begin{enumerate}
    \item $x \in \mathrm{Disc}_{\mathbb{F}}(s_N, K)$
    \item $x \notin \mathrm{Disc}_{\mathbb{F}}(s_N, K)$
\end{enumerate}
If $x \in \mathrm{Disc}_{\mathbb{F}}(s_N, K)$, we choose $\delta'_x > 0$ such that $(x, x+\delta'_x) \cap \mathrm{Disc}_{\mathbb{F}}(s_N, K) = \emptyset$.

If $x \notin \mathrm{Disc}_{\mathbb{F}}(s_N, K)$, there exists $a,b \in\mathrm{Disc}_{\mathbb{F}}(s_N, K)$ such that $ x = (a,b)$. By picking $C = K \cap (a,b)$, we get $\restr{s_N}{C} \equiv c_x$, a constant. Moreover, to ensure that $s_n$ remains constant in some neigbourhood: $(x, x+ \delta''_x)$, we let:
$$\delta''_x = \frac{1}{2}\min\left\{|x - a|_{\mathbb{R}}, |x - b|_{\mathbb{R}}\right\}$$
such that:
$$\restr{s_N}{K \cap (x, x+\delta''_x)} \equiv c_x$$
Choosing $\delta_x = \min\{\delta_x',\delta_x''\}$ gives us:
$$ \restr{s_N}{K \cap (x, x + \delta_x)} = c_x$$
By considering two arbitrary elements $t,u \in (x, x + \delta_x)$ we get the following inequality:
\begin{align*}
    |f(t) - f(u)|_{\mathbb{F}} 
    &= |\left(f(t) - s_N(t)\right) +\left(s_N(t) - s_N(u)\right) + \left(s_N(u) - f(u)\right)|_{\mathbb{F}}\\
    &\leq |f(t) - s_N(t)|_{\mathbb{F}} + |s_N(t) - s_N(u))|_{\mathbb{F}} + |s_N(u) - f(u)|_{\mathbb{F}}\\
    &\leq \frac{\epsilon}{2} + \frac{\epsilon}{2}\\
    &=\epsilon
\end{align*}
Here, we note that the difference between the two simple functions cancels out.

This inequality implies that any sequence $\{y_n\}_{n \in \mathbb{N}} \subseteq K$ such that $y_n  \xrightarrow[\;n\to\infty\;]{|\cdot|_{\mathbb{F}}}x^+$ will yield a Cauchy sequence  $\{f(y_n)\}_{n \in \mathbb{N}}$. 

As $(\mathbb{F}, |\cdot|_{\mathbb{F}}, \mathbb{F})$ is a Banach space, we have $f(y_n)  \xrightarrow[\;n\to\infty\;]{|\cdot|_{\mathbb{F}}} f(x^+)$ and $f(x^+)\in \mathbb{F}$.

The existence of left limits is treated analogously. This proves the reverse direction.

For $(1\iff 2) \implies 3$, let $f \in \mathrm{Reg}_{\mathbb{F}}(K)$, which is uniformly approximable by the following sequence of step functions $\{s_n\}_{n \in \mathbb{N}}$. Let $N \in \mathbb{N}$ such that for all $n  \geq N :\|f - s_N\|_{\infty} < \frac{1}{5n}$.
By definition of a regulated function, it is straightforward to show that $f$ admits only removable and jump discontinuities. Hence: 
$$\mathrm{Disc}_{\mathrm{F}}(f, K) = \bigcup_{n=1}^\infty\left\{x \in K \mid  \max\{|f(x) - f(x^-)|_{\mathbb{F}},|f(x) - f(x^+) |_{\mathbb{F}}, |f(x^-) - f(x^+)|_{\mathbb{F}} \geq \frac{1}{n}\right\}$$
For the sake of brevity, let:
$$\mathrm{D}_n: = \left\{x \in K \mid  \max\{|f(x) - f(x^-)|_{\mathbb{F}},|f(x) - f(x^+) |_{\mathbb{F}}, |f(x^-) - f(x^+)|_{\mathbb{F}} \geq \frac{1}{n}\right\}$$
Fix $x \in \mathrm{D}_n$. If $x$ is a jump discontinuity, we choose $\delta_x > 0$ such that:
$$|f(t) - f(x^+)|_{\mathbb{F}} < \frac{1}{5n} \quad \text{for all } t \in(x, x+\delta_x) \cap K$$
$$|f(u) - f(x^-)|_{\mathbb{F}} < \frac{1}{5n} \quad \text{for all } u \in(x - \delta_x, x) \cap K$$
But then:
\begin{align*}
    |s_N(t) - s_N(u)|_{\mathbb{F}} 
    &\geq |f(t) - f(u)|_{\mathbb{F}} - |s_N(t) - f(t)|_{\mathbb{F}} - |s_n(u) - f(u)|_{\mathbb{F}}\\
    &\geq |f(x^+) - f(x^-)|_{\mathbb{F}} - |f(u) - f(x^-)|_{\mathbb{F}} - |f(t) - f(x^+)|_{\mathbb{F}} - \frac{2}{5n}\\
    &\geq \frac{1}{5n} 
\end{align*}
This implies: $\mathrm{Disc}_{\mathbb{F}}(s_N,K) \cap ( x - \delta_x, x + \delta_x) \neq \emptyset$.

If $x$ is a removable discontinuity, set $L = f(x^-) = f(x^+)$. 

Then:
$$|f(x) - L|_{\mathbb{F}} > \frac{1}{n}$$
By choosing $\delta_x > 0$ such that:
$$|f(t) - L|_{\mathbb{F}} < \frac{1}{5n} \quad \text{for all } t \in K \cap \left(( x - \delta_x, x + \delta_x) \setminus \{x\}\right)$$
we assume (for contradiction) that $s_N$ has no discontinuities in $K \cap (x - \delta_x, x + \delta_x)$. 

Set $\restr{s_N}{K \cap (x - \delta_x, x + \delta_x)} \equiv c_x$. Then, for any $t \in K \cap (x - \delta_x, x + \delta_x)$:
$$|c_x - L|_{\mathbb{F}} \leq |c_x - s_N(t)|_{\mathbb{F}} + |s_N(t) - f(t)|_{\mathbb{F}} + |f(t) - L|_{\mathbb{F}} \leq \frac{2}{5n}$$
$$|f(x) - c_x|_{\mathbb{F}} \geq |f(x) - L |_{\mathbb{F}} + |L - c_x|_{\mathbb{F}} > \frac{1}{n} - \frac{2}{5n} > 0$$
This contradicts the fact that $f$ is uniformly approximable by step functions.

Hence: $\mathrm{Disc}_{\mathbb{F}}(s_N, K) \cap (x - \delta_x, x + \delta_x) \neq \emptyset$. This proves the following fact:
$$\text{For all } n\in \mathbb{N} \text{ and }x \in \mathrm{D}_n, \text{ there exists } (x - \delta_x, x + \delta_x): \mathrm{Disc}_{\mathbb{F}}(s_N, K) \cap (x - \delta_x, x + \delta_x) \neq \emptyset$$
For any $x \in \mathrm{D}_n$, we choose:
$$\rho_x : = \min\left\{\delta_x, \frac{1}{2}d\left(x, \mathrm{Disc}_{\mathbb{F}}(s_N, K) \setminus \{x\}\right)\right\}$$
This threshold gives us a family of disjoint intervals $\{(x - \rho_x, x + \rho_x) \mid 
 x \in \mathrm{D}_n\}$ such that:
$$\text{For all } x \in \mathrm{D}_n: (x - \rho_x, x + \rho_x)\cap \mathrm{D}_{\mathbb{F}}(s_N, K) \neq \emptyset$$
This implies that the following sequence of maps has every term injective:
$$\theta_n : \mathrm{D}_n \longmapsto \mathrm{Disc}_{\mathbb{F}}(s_N, K), \quad x \longmapsto (x -\rho_x, x+\rho_x) \cap \mathrm{Disc}_{\mathbb{F}}(s_N, K)$$
This finishes the forward direction.

For $3 \implies (1\iff 2)$, we simply observe that unpacking the two definitions of a jump and a removable discontinuity gives us the definition of a regulated function.  \\ This proves our equivalent formulations.
\end{proof}
\bigskip

We now move on to two important properties of regulated functions, which will be used later on.

\begin{proposition}[Boundedness]
    Any regulated function $f\in \mathrm{Reg}_{\mathbb{F}}(K)$ is bounded.
\end{proposition}

\begin{proof}
We consider an arbitrary regulated function $f\in \mathrm{Reg}_{\mathbb{F}}(K)$. By definition, for any fixed $x \in K$, we can choose some $\delta_x > 0$ such that:
$$|f(t)-f(x^-)|_{\mathbb{F}} < 1 \quad \text{ for all } t\in K \cap (x - \delta'_x, x)$$
$$|f(t) - f(x^+)|_{\mathbb{F}} <1 \quad \text{ for all } t \in K \cap (x, x+\delta''_x)$$
Since the values of $f$ on $K \cap (x - \delta_x, x + \delta_x)$ are bounded by:
$$M_x: = \max\left\{|f(x)|_{\mathbb{F}}, |f(x^+)|_{\mathbb{F}} + 1, |f(x^-)|_{\mathbb{F}} + 1\right\}$$
we get:
$$\text{For every } x \in K, \text{there exists } U_x:= (x-\delta_x, x +\delta_x) \text{ such that } |\restr{f}{U_x}|_{\mathbb{F}} \leq M_x$$
As $\{U_x \mid x \in K\}$ forms an open cover of $K$, we can select a finite open subcover $\left\{U_{x_i}\mid i= \overline{1,n}\right\}$. By taking $M = \max\limits_{i=1,n}M_{x_i}$, we conclude:
$$\|f\|_{\infty} \leq M$$
This proves the boundedness property.
\end{proof}
\newpage
\begin{proposition}[Borel-Measurability] Any regulated function $f \in \mathrm{Reg}_{\mathbb{F}}(K)$ is \\$(Bor(K),Bor(\mathbb{F}))$-measurable.
\end{proposition}
\begin{proof}
Let $f \in \mathrm{Reg}_{\mathbb{F}}(K)$. We consider partitioning $K$ as follows:
$$K = \mathrm{Cont}_{\mathbb{F}}(f,K) \sqcupscaled[1.2]{0pt} \mathrm{Disc}_{\mathbb{F}}(f,K)$$
As the set of discontinuities of a regulated function is at most countable, we have that both $\mathrm{Cont}_{\mathbb{F}}(f, K)$   and  $\mathrm{Disc}_{\mathbb{F}}(f, K)$ are in $Bor(K)$. Additionally, the restriction of our function to the set of continuous points $\restr{f}{\mathrm{Cont}_{\mathbb{F}}(f, K)}$ is a continuous function. This means that for every open set  $U\subseteq \mathbb{F}$ there exists some open set (with respect to the subspace topology) $V \subseteq K$ such that:
$$\left(\restr{f}{\mathrm{Cont}_{\mathbb{F}}(f, K)}\right)^{-1}\left(U\right) = V \cap\mathrm{Cont}_{\mathbb{F}}(f, K)$$
Moreover, we arrive at the following set-theoretical identity:
$$\left(\restr{f}{\mathrm{Cont}_{\mathbb{F}}(f, K)}\right)^{-1}\left(U\right) = \left\{x \in \mathrm{Cont}_{\mathbb{F}}(f, K)\mid f(x) \in U \right\} = f^{-1}(U) \cap \mathrm{Cont}_{\mathbb{F}}(f,[a,b])$$
We now check the measurability condition:
$$\text{For any open set } U \subseteq F: f^{-1}(U) = \left(f^{-1}(U) \cap \mathrm{Cont}_{\mathbb{F}}(f, K)\right) \,\,\sqcupscaled[1.2]{0pt}\,\, \left(f^{-1}(U) \cap \mathrm{Disc}_{\mathbb{F}}(f, K)\right) $$
By analysing the two disjoint sets, we observe:
$$f^{-1}(U) \cap \mathrm{Cont}_{\mathbb{F}}(f,K) = \left(\restr{f}{\mathrm{Cont}_{\mathbb{F}}(f,K)}\right)^{-1}\left(U\right) \in Bor(K)$$
$$f^{-1}(U) \cap \mathrm{Disc}_{\mathbb{F}}(f,K) \subseteq \mathrm{Disc}_{\mathbb{F}}(f,K)$$
Since the set of discontinuities of our function is at most countable, so will the subset $f^{-1}(U) \cap \mathrm{Disc}_{\mathbb{F}}(f,K)$ be at most countable. This implies that it is also in $Bor(K)$.
Thus, our measurability test has been checked:
$$\text{For any open subset } U \subseteq \mathbb{F}: f^{-1}(U) \in Bor(K)$$
This proves measurability.
\end{proof}

%% file: 5-HK.tex
\chapter{The Bounded Henstock-Kurzweil Functional Calculus}\label{ch:henstockkurzweilbounded}
In this chapter, we will provide a new construction for a particular functional calculus, verify its validity, prove that it is an extension of the continuous functional calculus, and extend the spectral mapping theorem to regulated functions.

We have chosen this name for our functional calculus simply because our construction mimics the way in which Ralph Henstock and Jaroslav Kurzweil have generalized the Riemann integral, leading to the broadest way of integrating functions to date. 

The reader who wants to see the construction of this integral can consult Robert Bartle's \enquote{A Modern Theory of Integration}\,\cite{Bartle}.

\bigskip
\bigskip 

Before we introduce our notion of integrability, we will need a sum that will approach the integral. For the rest of this chapter, we will fix an arbitrary linear, bounded, self-adjoint operator $A$.
\begin{definition}[The underlying complex measures $\mu_{x,y}$]
    Given a  Hilbert space $(H,\mathbb{F})$, a self-adjoint operator $A \in \mathrm{B}(H)$, and the (unique) projection-valued $E$ measure from the Borel functional calculus, we will associate to our \textbf{$\mu_{x,y}-HK$ sums} a family of complex measures $\{\mu_{x,y}\}_{(x,y)\in H\times H}$, defined as follows:
    $$\text{For all } x,y \in H \text{ and } I\in Bor(\sigma(A)):  \mu_{x,y}(I): = \inner{E(I)x}{y}_H$$

\end{definition}
\begin{definition}[The $\mu\text{-}HK$ Sum]
Given a Hilbert space $(H,\mathbb{F})$, a regulated function \\ $f \in \mathrm{Reg}_{\mathbb{F}}(\sigma(A))$, a gauge-fine tagged partition $\widehat{P}_{\gamma} := \left\{\left(t_i, I_i, \gamma\right)\mid i=\overline{1,n}\right\}$ of $\sigma(A)$, and the (unique) projection-valued measure $E$ from the Borel functional calculus , we define the \textbf{$\mu_{x,y}\text{-}HK$ sum} of $f$ as follows:
\begin{align*}
    \text{For all } x,y \in H : S(f,\mu_{x,y},\widehat{P}_\gamma)&:=\sum\limits_{i=1}^nf(t_i)\mu_{x,y}(I_i) \text{ where each } I_i \subseteq(t_i - \delta(t_i), t_i + \delta(t_i))\\
\end{align*}
\end{definition}
It is now natural to consider the continuous variant of this sum by letting our partition become arbitrarily fine. This is where we introduce the notion of $\mu_{x,y}\text{-}HK$ integrability.
\begin{definition}[$\mu\text{-}HK$ integrability]
A function $f:\sigma(A)\to\mathbb{F}$ is \textbf{$\mu_{x,y}\text{-}HK$-integrable} with integral $I_{x,y}(f):=(HK)\displaystyle\int\limits_{\sigma(A)}f(\lambda)d\mu_{x,y}$ if:
$$\left(\forall\epsilon >0\right)\left(\exists\gamma\text{ a gauge function}\right)\left(\forall x,y \in H\right)\left(\exists I_{x,y} \in \mathbb{F}\right)$$$$\left(\forall \text{gauge-fine, tagged partitons } \widehat{P_\gamma}\implies \left|S(f,\mu_{x,y},\widehat{P}_\gamma) - I_{x,y}(f)\right|_{\mathbb{F}} <\epsilon\right)$$
\end{definition}
To start our construction, we leverage the fact that every regulated function can be approximated uniformly by step functions. This means that we can begin our $\mu\text{-}HK$ functional calculus on step functions and then extend it to all regulated functions.

Let  us fix an arbitrary step function on $\sigma(A)$:
$$s = \sum\limits_{i=1}^na_i\mathds{1}_{C_i}$$
where:
$$\text{For all } i = \overline{1,n}: C_i = K \cap [p_{i-1},p_i)  \text{ with } \{p_i\}_{i = \overline{0,n}}  \, \text{ an exhaustive } K\text{-division} $$
We consider the following gauge function:
$$\gamma(t):= 
\begin{cases}
    \frac{1}{2}d(t, K \setminus C_i), &\text{for } t \in C_i: C_i \text{ not a singleton}\\
    \gamma(t): K\cap (t - \gamma(t), t + \gamma(t)) = \{t\}, &\text{for } t \in C_i =\{t\}
\end{cases}$$
Let $\widehat{P}_{\gamma}: =\left\{\left(t_j,C'_j, \gamma)\right)\mid j=\overline{1,n}\right\}$ be an arbitrarily chosen gauge-fine, tagged partition. By the way in which we have chosen our gauge function, each $J_j$ will lie entirely in one $C_i$. 

Hence: 
$$ s(t_j) = a_{i(j)} \text{ whenever } t_j \in C'_j \subseteq C_i$$
We will collect all the indices that are characterised by this inclusion property:
$$\Omega_{j,i}:= \{j \mid C'_j \subseteq C_{i}\}$$
One can prove with simple set-theoretic arguments that the following identity holds:
$$\bigcup\limits_{j \in \Omega_{j,i}} C'_j = C_i$$
We claim that for all $x,y \in H:I_{x,y}(s)=  \displaystyle\sum\limits_{i=1}^na_i\,\mu_{x,y}(C_i)$

By the definition of the $\mu_{x,y}\text{-}HK$ sum, we get:
\begin{align*}
    S(s, \mu_{x,y}, \widehat{P}_{\gamma}) 
    &= \sum \limits_{j=1}^ns(t_j)\,\mu_{x,y}(C'_j)\\
    &=\sum\limits_{i=1}^na_{i(j)}\left(\sum\limits_{j\in \Omega_{j,i}}\mu_{x,y}(C'_j)\right)\\
    &=\sum\limits_{i=1}^n a_{i(j)} \,\mu_{x,y}\left(\bigcup\limits_{j\in \Omega_{j,i}}C'_j\right)\\
    &=\sum\limits_{i=1}^n a_i\,\mu_{x,y}(C_i)\\
    &= I_{x,y}(s)
\end{align*}
This establishes the functional calculus for step functions.
\begin{corollary}[The $\mu\text{-}HK$ integral of step functions]
Given a Hilbert space $(H,\mathbb{F})$, a step function $s = \displaystyle\sum\limits_{i=1}^na_i\mathds{1}_{C_i}$ on $\sigma(A)$, and the (unique) projection-valued measure $E$ from the Borel functional calculus , its $\mu_{x,y}\text{-}HK$ integral is:
$$I_{x,y}(s)= \sum\limits_{i=1}^n c_i \, \mu_{x,y}(C_i) \text{ for all } x.,y \in H$$
\end{corollary}
In here, we leverage the fact that the $\mu_{x,y}\text{-}HK$ integral of a step function can be seen as a bounded sesquilinear form on $(H,\mathbb{F})$ with respect to $x,y \in H$. As a consequence, there exists a unique bounded operator $s(A)$ such that:
$$\text{For  all } x,y\in H: \inner{s(A)x}{y}_H = I_{x,y}(s) $$
This remark will turn out to be crucial when verifying the axioms of a functional calculus.
\bigskip

We can now extend the functional calculus to all regulated functions through their uniform approximability by step functions. 
We begin by fixing an arbitrary regulated function $f \in \mathrm{Reg}_{\mathbb{F}}(\sigma(A))$ and a sequence of step functions $\{s_n\}_{n \in \mathbb{N}}$ converging uniformly to $f$. The choice of our gauge function will remain the same as the one we constructed for the simple functions. Without loss of generality, which will be apparent after the entire argument, we will fix $x,y \in \partial B(0,1) \subseteq H$. 
By considering an arbitrarily gauge-fine, tagged partition $\widehat{P}_{\gamma}: =\left\{\left(t_j,C_j, \gamma)\right)\mid j=\overline{1,n}\right\}$, we have:
\begin{align*}
    \hspace{-1cm}\left|S(f, \mu_{x,y}, \widehat{P}_{\gamma})  - \int\limits_{\sigma(A)}s_n(\lambda)\,d\mu_{z,y}(\lambda)\right|_{\mathbb{F}} 
    &= \left|S((f-s_n) + s_n, \mu_{x,y}, \widehat{P}_{\gamma}) - \int\limits_{\sigma(A)}s_n(\lambda)\, d\mu_{x,y}(\lambda)\right|_{\mathbb{F}} \\
    &=\left|\left(S(s_n, \mu_{x,y}, \widehat{P}_{\gamma}) - \int\limits_{\sigma(A)}s_n(\lambda)d\mu_{x,y}(\lambda)\right) + S(f-s_n, \mu_{x,y}, \widehat{P}_{\gamma}) \right |_{\mathbb{F}}   \\
    &\leq \left|(S(s_n, \mu_{x,y}, \widehat{P}_{\gamma}) - \int\limits_{\sigma(A)}s_n(\lambda)\,d\mu_{x,y}(\lambda) \right|_{\mathbb{F}}     + \left|S(f-s_n, \mu_{x,y}, \widehat{P}_{\gamma}) \right|_{\mathbb{F}}\\
    &\leq\|f -s_n\|_{\infty} \sum\limits_{i=1}^n\left|\mu_{x,y}(C_i)\right|_{\mathbb{F}}\\
    &\leq\|f -s_n\|_{\infty}\|x\|_H\|y\|_H\\
\end{align*}

By taking the limit as $n$ approaches infinity on  both sides, we will obtain the following result:
$$I_{x,y}(f) = \lim_{n \to \infty}\int\limits_{\sigma(A)}s_n(\lambda)\,d\mu_x,y(\lambda)$$
One can show that this limit value is well defined, in the sense that it does not depend on the sequence approaching it.

If either $x = 0_H$ or $y = 0_H$, we will have that the sum (and hence, the integral) evaluates to zero. Let us consider two arbitrary non-zero points $x,y \in H: x = \|x\|_H \, \hat{x}, \, y = \|y\|_H \, \hat{y}$, where $\hat{x}, \hat{y} \in \partial B(0,1)$.
By the sesquilinearity  of the $\mu_{x,y}\text{-}HK$ sum with respect to $x,y$, we get:
$$S(f, \mu_{x,y},\widehat{P}_{\gamma}) = \|x\|_H \|y\|_H S(f, \mu_{\hat{x},\hat{y}},\widehat{P}_{\gamma})$$
By an identical argument:
$$ I_{x,y}(f) = \|x\|_H \|y\|_H \,\, I_{\hat{x},\hat{y}}(f)$$
\newpage
We turn once again to the power of bounded, sesquilinear forms. As regulated functions are bounded, we will get, in particular: 
$$|I_{x,y}(f)|_{\mathbb{F}} \leq \|f\|_{\infty}\|x\|_H\|y\|_H < \infty$$
This shows that there exists a unique bounded, linear operator $f(A) \in \mathrm{B(H)}$ such that:
$$\text{For all } x,y \in H:\inner{f(A)x}{y}_H = I_{x,y}(f)$$
\begin{corollary}[The $\mu\text{-}HK$ integral of regulated functions]
Given a Hilbert space $(H,\mathbb{F})$, a regulated function $f\in \mathrm{Reg}_{\mathbb{F}}(\sigma(A))$, any sequence $\{s_n\}_{n \in \mathbb{N}}$ uniformly converging to $f$, and the (unique) projection-valued measure from $E$ the Borel functional calculus, the $\mu_{x,y}\text{-}HK$ integral of $f$ is:
$$I_{x,y}(f) = \lim\limits_{n \to \infty}\int\limits_{\sigma(A)}s_n(\lambda)\,d\mu_{x,y}(\lambda) \text{ for all } x,y \in H$$
\end{corollary}
In verifying the axioms of a functional calculus, we will use the following lemma:
\begin{lemma}[Lipschitz continuity]\label{lem:lipschitzcontinuity}
Given a Hilbert space $(H,\mathbb{F})$, two regulated functions $f_1, f_2\in \mathrm{Reg}_{\mathbb{F}}(\sigma(A))$, and the (unique) projection-valued measure $E$ from the Borel functional calculus , we have:
$$\|(f_1-f_2)(A)\|_{op} \leq \|f_1-f_2\|_{\infty}$$
\end{lemma}
\begin{proof}
We fix $x,y\in H$. As $(\mathrm{Reg}_{\mathbb{F}}(\sigma(A), \cdot_{\mathrm{Reg}_{\mathbb{F}}(\sigma(A)},\mathbb{F})$ is an associative algebra with respect to function multiplication, we have, in particular, that the difference of regulated functions yields a regulated function. Hence:
$$|\inner{(f-g)(A)x}{y}_H|_{\mathbb{F}} \leq \|f-g\|_{\infty} \|x\|_H\|\|y\|_H$$
By the bijection between bounded sesquilinear forms and bounded, linear operators, we have:
$$\|(f-g)(A)\|_{op} = \sup_{\|x\|_H = \|y\|_H =1}\|\inner{(f-g)(A)x}{y}_H\|$$
Thus, by taking the supremum over all unit-normed $x,y \in H$, we get:
$$\|(f-g)(A)\|_{op} \leq \|f-g\|_{\infty}$$
This finishes the proof.
\end{proof}
\newpage
\begin{theorem}[The $HK$ Functional Calculus is a bona fide Functional Calculus]
Given the $*$-algebra of regulated functions $(\mathrm{Reg}_{\mathbb{F}}(\sigma(A)),\mathbb{F})$, the following map is a $*$-algebra homomorphism:
\[
  \Phi^{HK}_A\colon
  \mathrm{Reg}_{\mathbb{F}}(\sigma(A))
  \;\longrightarrow\;
  \mathrm{B(H)},
  \quad
  f\;\longmapsto\;f(A) 
\]
\end{theorem}
\begin{proof}
We will verify the following axioms in the prescribed order:

For all $f_1,f_2,f_3 \in \mathrm{Reg}_{\mathbb{F}}(\sigma(A))$ and $\alpha_1,\alpha_2 \in \mathbb{F}$:
\begin{enumerate}[label=\arabic*)]
\item $\Phi^{HK}_A(\alpha_1 f_1 + \alpha_2f_2) = \alpha_1\Phi^{HK}_A(f_1) + \alpha_2 \Phi^{HK}_A(f_2)$
\item $\Phi^{HK}_A(f_1f_2) = \Phi^{HK}_A(f_1)\Phi^{HK}_A(f_2)$
\item $\Phi^{HK}_A(\overline{f_3}) = (\Phi_A^{HK}(f_3))^*$
\item $\Phi^{HK}_A(1_{\mathrm{Reg}_{\mathbb{F}}(\sigma(A))}) = 1_{\mathrm{B}(H)}$
\end{enumerate}
\bigskip
We start off by fixing three sequences of step functions: $$\left\{s^{f_1}_n\right\}_{n \in \mathbb{N}} \qquad \left\{s^{f_2}_n\right\}_{n \in \mathbb{N}} \qquad \left\{s^{f_3}_n\right\}_{n \in \mathbb{N}}$$
each one converging uniformly to the superscripted function. Throughout all the properties, we fix $x,y \in H$.

For $1)$, we just use the linearity of the sum, giving us immediately the result:
\begin{align*}
    \inner{(\alpha_1f_1+\alpha_2f_2)(A)x}{y}_H &= \int\limits_{\sigma(A)}(\alpha_1f_1 + \alpha_2f_2)(\lambda)\,d\mu_{x,y}(\lambda)\\
    &= \alpha_1\hspace{-0.3cm}\int\limits_{\sigma(A)}f_1(\lambda)\,d\mu_{x,y}(\lambda) + \alpha_2\hspace{-0.3cm}\int\limits_{\sigma(A)}f_2(\lambda)\,d\mu_{x,y}(\lambda)\\
    &= \alpha_1\inner{f_1(A)x}{y}_H + \alpha_2\inner{f_2(A)x}{y}_H
\end{align*}
This proves the first property.
\newpage

For $2)$, we will make use of our previously defined step functions.

Suppose:
$$ s^{f_1}_n = \sum\limits_{i=1}^{K_n}b_i \mathds{1}_{I_i} \qquad s_n^{f_2} = \sum\limits_{i=1}^{M_n}b_i \mathds{1}_{J_i} \qquad\text{ for all } n\in \mathbb{N}$$
Then: 
$$ \inner{s^{f_1}_n(A)x}{y}_H = \sum\limits_{i=1}^{K_n}b_i \mu_{x,y}(I_i) \qquad \inner{s_n^{f_2}(A)x}{y}_H = \sum\limits_{i=1}^{M_n}c_i 
\, \mu_{x,y}(J_i) \qquad\text{ for all } n\in \mathbb{N}$$
By using the multiplicativity of projection-valued measures: 
$$\text{For all }  B_1,B_2 \in Bor(\sigma(A)): E(A \cap B) = E(A)E(B)$$
we get:
\begin{align*}
     \inner{s^{f_1}_n(A)x}{y}_H\inner{s_n^{f_2}(A)x}{y}_H  
     &= \sum \limits_{i,j}b_i c_j \,\mu_{x,y}(I_i)\,\mu_{x,y}(J_j)\\
     &=\sum \limits_{i,j}b_i c_j \,\mu_{x,y}(I_i \cap J_j)\\
     &=\inner{\left(s^{f_1}_ns_n^{f_2}\right)(A)x}{y}_H
\end{align*}
This translates, operatorially speaking, to: $s_n^{f_1}(A) s_n^{f_2}(A) = \left(s_n^{f_1}s_n^{f_2}\right)(A)$ for all $n \in \mathbb{N}$.
Upon inspecting the convergence of both of these operator formulations, we get:
\begin{align*}
s_n^{f_1}(A) s_n^{f_2}(A)  &\xrightarrow[\;n\,\to\,\infty\;]{\|\cdot\|_{op}} f_1(A)f_2(A)\\
\left(s_n^{f_1} s_n^{f_2}\right)(A)  &\xrightarrow[\;n\,\to\,\infty\;]{\|\cdot\|_{op}} \left(f_1(A)f_2(A)\right)   
\end{align*}
As the two expressions on the left-hand side coincide, we get our final result:
$$\left(f_1 f_2\right)(A) = f_1(A)f_2(A)$$
This proves the second property.

\newpage
For $3)$, the property is easily derived once we establish the following remark. 
\begin{align*}
  \hspace{-2cm} \text{For all } B\in Bor(\sigma(A)), \text{ we have : } \mu_{x,y}(B)
    &= \inner{E(B)x}{y}_H \\
    &=\overline{\inner{y}{E(B)x}_H}\\
    &=\overline{\inner{E(B)y}{x}_H}\\
    &=\overline{\mu_{y,x}(B)}
\end{align*}
Hence:
\begin{align*}
    \inner{\overline{f}(A)x}{y}_H 
    &= \int\limits_{\sigma(A)}\overline{f(\lambda)}\,d\mu_{x,y}(\lambda)\\
    &= \int\limits_{\sigma(A)}\overline{f(\lambda)\,d\mu_{y,x}(\lambda)}\\
    &= \overline{\int\limits_{\sigma(A)}f(\lambda)\,d\mu_{y,x}(\lambda)}\\
    &= \overline{\inner{f(A)y}{x}}_H \\
    &= \inner{x}{f(A)y}_H \\
    &= \inner{(f(A))^*x}{y}_H 
\end{align*}
By using the bijection between bounded sesquilinear forms and bounded, linear operators, we get: $$(f(A))^* = \overline{f}(A)$$
This proves the third property. 

\bigskip
For $4)$, we remark that for any gauge-fine, tagged partition $\widehat{P}_{\gamma}:= \left\{\left(t_i, I_i, \gamma\right)\mid i=\overline{1,n}\right\}$ (where the gauge $\gamma$ is the same one we used for simple functions) we have:
$$I_{x,y}((1_{\mathrm{Reg}_{\mathbb{F}}(\sigma(A))}) = \sum\limits_{i=1}^n\mu_{x,y}(I_i) = \inner{E(\sigma(A)x}{y}_H = \inner{x}{y}_H = \inner{I(x)}{y}_H$$
Since $I \in \mathrm{B(H)}$ is also unique, by the bijection between bounded sesquilinear forms and bounded, linear operators,  we get:
$$1(A) = I$$
This proves the fourth property, establishing the fact that our construction is a bona fide functional calculus.
\end{proof}

\newpage
Since the space of regulated functions is strictly larger than the space of continuous functions, we will check if the $HK$ functional calculus agrees with the continuous functional calculus. If we prove this fact, then the $HK$ functional calculus can be seen as a well-defined extension of the smaller functional calculus.
\begin{theorem}[The \texorpdfstring{$HK$}{TEXT} functional calculus extends the continuous functional calculus]
Given a Hilbert space $(H,\mathbb{F})$, a self-adjoint operator $A \in \mathrm{B(H)}$, and the (unique) projection-valued measure $E$ from the Borel functional calculus, we get the following agreement on $\mathrm {C_{\mathbb{F}}}(\sigma(A))$:
$$\Phi^{HK}_A = \Phi^C_A $$
\end{theorem}
\textbf{Proof:}
We leverage the fact that any continuous function on a compact set can be uniformly approximated by step functions. Let $f \in \mathrm {C_{\mathbb{F}}}(\sigma(A))$ and $\left\{s_n^f\right\}_{n \in \mathbb{N}}$ be the sequence of step functions uniformly approximating $f$. Since Lipschitz continuity holds in the Borel functional calculus as well, we have:
$$s_n^f  \xrightarrow[\;n\,\to\,\infty\;]{\|\cdot\|_{\infty}}  f \,\,\,\text{ implies } \,\,\, 
\begin{cases}
    \Phi^{HK}_A(s_n^f) \xrightarrow[\;n\,\to\infty\;]{\|\cdot\|_{op}}  \Phi^{HK}_A(f)\\
    \Phi^{B}_A(s_n^f) \xrightarrow[\;n\,\to\infty\;]{\|\cdot\|_{op}}  \Phi^{B}_A(f)
\end{cases}$$
Since limits are unique in Hilbert spaces, we get:
$$\Phi^{HK}_A(f) = \Phi^B_A(f) = \phi_A^C(f)$$
As $f \in \mathrm {C_{\mathbb{F}}}(\sigma(A))$  was arbitrarily chosen, we conclude:
$$\Phi^{HK}_A = \Phi^C_A \quad\text{ on } \,\mathrm {C_{\mathbb{F}}}(\sigma(A))$$
This proves the claim. 

Moreover, the more general fact can be proven by the exact same reasoning applied to approximating regulated functions by step functions:
$$\Phi^{HK}_A = \phi^{B}_A \quad\text{ on } \, \mathrm{Reg}_{\mathbb{F}}(\sigma(A))$$
This not only suggests that the Henstock-Kurzweil functional calculus is an extension of the continuous functional calculus, but it is precisely the Borel functional calculus restricted to the space of regulated functions.

%% file: 6-TOP.tex
\chapter{The Advantage of the \texorpdfstring{$HK$}{TEXT} Functional Calculus} \label{ch: advantageofhk}

In this chapter, we aim to provide the main advantage of the $HK$ functional calculus in contrast with the Borel functional calculus. By the defining property of regulated functions, every such function can be uniformly approximated by step functions. Combining this property with the Lipschitz continuity of the $HK$ functional calculus opens up the possibility of allowing operators to converge in the operator norm. As we will see, this is not always the case when examining functions in a larger space. 
\bigskip

We will begin by realising the space of regulated functions as a $\|\cdot\|_{\infty}$-closed subalgebra of $(\mathrm{Borel}_{b,\mathbb{F}}(\sigma(A)), \|\cdot\|_{\infty}, \mathbb{F})$.

We will first prove that $(\mathrm{Reg}_{\mathbb{F}}(\sigma(A)), \|\cdot\|_{\infty}, \mathbb{F})$ is a Banach space. 
\begin{theorem}{$(\mathrm{Reg}_{\mathbb{F}}(\sigma(A)), \|\cdot\|_{\infty}, \mathbb{F})$ is a Banach Space.}\end{theorem}
\begin{proof}
One can easily show that under the following operations:
\begin{align*}
    (f+g)(x)&: = f(x) + g(x)\\
    (\alpha f)(x)&:=\alpha f(x)
\end{align*}
The space of $\mathbb{F}$-valued regulated functions becomes a vector space.

To show completeness, we consider an arbitrary Cauchy sequence $\{f_n\}_{n \in \mathbb{N}}\subseteq \mathrm{Reg}_{\mathbb{F}}(\sigma(A))$. Since $(\mathrm{Borel}_{b,\mathbb{F}}(\sigma(A)), \|\cdot\|_{\infty}, \mathbb{F})$ is a Banach space, we have that our Cauchy sequence converges uniformly to some bounded, measurable function $f$.

We fix $\epsilon > 0$ and pick a threshold $N_1 \in \mathbb{N}$ such that:
$$\text{For all } n \geq N_1: \|f_n - f\|_{\infty} < \frac{\epsilon}{3} $$

What remains to be shown is that $f$ itself is regulated.
\newpage
Since the argument for the existence of right limits is the same as the argument for the existence of left limits, we will focus on proving only the former. 
By the fact that our sequence $\{f_n\}_{n \in \mathbb{N}}$ consists of regulated functions, we have:
 $$\left(\forall \epsilon > 0\right)\left(\exists \delta_x > 0\right)\left(\forall y \in K \cap (x, x+\delta_x) \implies |f_n(y) - f_n(x^+)|_{\mathbb{F}} < \frac{\epsilon}{3}\right)$$
 Since $\{f_n\}_{n \in \mathbb{N}}$ is Cauchy, $\{f_n(x^+)\}_{n \in \mathbb{N}}$ will be Cauchy in 
 $(\mathbb{F}, |\cdot|_{\mathbb{F}}, \mathbb{F})$. By the completeness of $(\mathbb{F}, |\cdot|_{\mathbb{F}}, \mathbb{F})$, we will have: 
 $$f_n(x^+) \xrightarrow[\,n \,\to\,\infty\,]{\,|\cdot|_{\mathbb{F}}\,} f(x^+) \,\,\text{ where } \,\, f(x^+)\in \mathbb{F}$$
 Concretely:
 $$\left(\forall \epsilon  > 0\right)\left(\exists N_2 \in \mathbb{N}\right)\left(\forall n \geq N_2  \implies |f_n(x^+) - f(x^+)|_{\mathbb{F}} < \frac{\epsilon}{3}\right)$$\
 Choosing $N_3: = \max\{N_1,N_2\}$ yields:
 \begin{align*}
     |f(y) - f(x^+)|_{\mathbb{F}} 
     &\leq   |f(y) - f_n(y)|_{\mathbb{F}} +   |f_n(y) - f_n(x^+)|_{\mathbb{F}} +   |f_n(x^+) - f(x^+)|_{\mathbb{F}}\\
     &< \frac{\epsilon}{3} + \frac{\epsilon}{3}  + \frac{\epsilon}{3}\\
     &= \epsilon
 \end{align*}
 This proves that $f(x^+) \in \mathbb{F}$, showing completeness of  $(\mathrm{Reg}_{\mathbb{F}}(\sigma(A)), \|\cdot\|_{\infty}, \mathbb{F})$. 
 
 With this, our proof is complete.
 \end{proof}
 \bigskip

We can now proceed with our isometric embedding. We note that our space of regulated functions is contained within the space of bounded, Borel-measurable functions. 

Hence, a natural choice of an isometric embedding is the canonical inclusion map.

\begin{proposition}$(\mathrm{Reg}_{\mathbb{F}}(\sigma(A)), \|\cdot\|_{\infty}, \mathbb{F})$ is isometrically embeddable into \\ $(\mathrm{Borel}_{b,\mathbb{F}}(\sigma(A)), \|\cdot\|_{\infty}, \mathbb{F})$. \end{proposition}
\begin{proof}
We consider the inclusion map:
$$\iota \colon\;\mathrm{Reg}_{\mathbb{F}}(\sigma(A))\;\longrightarrow\;\mathrm{Borel}_{b,\mathbb{F}}(\sigma(A)), \quad f\;\longmapsto\; \iota(f) $$
Since $\mathrm{Reg}_{\mathbb{F}}(\sigma(A)) \subsetneq   \mathrm{Borel}_{b,\mathbb{F}}(\sigma(A))$, our map is trivially well-defined. 
\newpage
As $(\mathrm{Reg}_{\mathbb{F}}(\sigma(A)), \cdot_{\mathrm{Reg}_{\mathbb{F}}(\sigma(A))}, \mathbb{F})$ is an algebra with respect to function multiplication, the axioms of a unital algebra homomorphism are satisfied.\\

For all $f_1. f_2 \in \mathrm{Reg}_{\mathbb{F}}(\sigma(A))$ and $\alpha \in \mathbb{F}$, we have:
\begin{enumerate}[label = \arabic*)]
    \item $\iota(f_1 + f_2) = f_1 + f_2  = \iota(f_1) + \iota(f_2)$
    \item $\iota(\alpha f_1) = \alpha f_1 = \alpha \iota(f_1)$ 
    \item $\iota(f_1f_2) = f_1 f_2 = \iota(f_1)\iota(f_2)$
    \item $\iota(1_{\mathrm{Reg}_{\mathbb{F}}(\sigma(A))}) = 1_{\mathrm{Reg}_{\mathbb{F}}(\sigma(A))} = 1_{ \mathrm{Borel}_{b,\mathbb{F}}(\sigma(A))}$
\end{enumerate}

\textbf{Isometry:}

By fixing  $f \in \mathrm{Reg}_{\mathbb{F}}(\sigma(A))$ we get:
$$\|\iota(f)\|_{\infty} = \sup\limits_{x \in \sigma(A)}|f(x)|_{\mathbb{F}} = \|f\|_{\infty}$$

This proves the isometry property.
\medskip

\textbf{Injectiveness:}

It is a general fact that linear isometries of normed vector spaces are injective. 

Explicitly, for all distinct $f_1, f_2 \in  \mathrm{Reg}_{\mathbb{F}}(\sigma(A))$, we have:
$$\|\iota(f_1) - \iota((f_2)\|_{\infty} = \|\iota(f_1 - f_2)\|_{\infty} = \|f_1 - f_2\|_{\infty}   > 0$$
This implies that $\iota(f_1) \neq \iota(f_2)$.
Hence, injectivity is proven.

\medskip

\textbf{Non-Surjectiveness:}

It suffices to find a non-regulated, bounded, measurable function. Since being non-regulated is equivalent to having uncountably many discontinuities, one immediate candidate is Dirichlet's function $\mathds{1}_{\mathbb{Q}}$. Since this indicator function is trivially bounded, we will prove that it is measurable.

Let us fix an open set $U  \subseteq \mathbb{F}$. There will be four mutually exclusive cases:
\begin{enumerate}[label = \arabic*)]
    \item $U \cap \{0,1\} = \emptyset$
    \item $\{0,1\} \subseteq U$
    \item $1 \in U$ but $0 \notin U$
    \item $0 \in U$ but $1 \notin U$
\end{enumerate}

\newpage
In this order, we will get:
\begin{enumerate}[label = \arabic*)]
    \item $\mathds{1}^{-1}_{\mathbb{Q}}(U) = \emptyset \in Bor(\sigma(A))$
    \item $\mathds{1}^{-1}_{\mathbb{Q}}(U) = \sigma(A) \in Bor(\sigma(A))$
    \item $\mathds{1}^{-1}_{\mathbb{Q}}(U) = \sigma(A) \cap \mathbb{Q}\in Bor(\sigma(A))$
    \item $\mathds{1}^{-1}_{\mathbb{Q}}(U) = \sigma(A) \cap \left(\mathbb{R} \setminus \mathbb{Q}\right) \subseteq Bor(\sigma(A))$
\end{enumerate}

This proves that Dirichlet's function is measurable and that our map is non-surjective. 

\medskip

As a consequence, we obtain the following isometric isomorphism, with respect to $\|\cdot\|_{\infty}:$

$$\iota \colon\;\mathrm{Reg}_{\mathbb{F}}(\sigma(A))\;\xrightarrow[]{\,\,\,\cong\,\,\,}\;\iota\left(\mathrm{Reg}_{\mathbb{F}}(\sigma(A))\right), \quad f\;\longmapsto\; \iota(f)$$

This proves our claim. 
\end{proof}
\medskip

In particular,  $(\mathrm{Reg}_{\mathbb{F}}(\sigma(A)), \cdot_{\mathrm{Reg}_{\mathbb{F}}(\sigma(A))}, \mathbb{F})$ can be seen as a closed $\|\cdot\|_{\infty}$ subalgebra of  \\ $(\mathrm{Borel}_{b,\mathbb{F}}(\sigma(A)), \|\cdot\|_{\infty}, \mathbb{F})$. With the help of this identification, we can also associate in a more aesthetic manner the relationship between the $HK$ and the Borel functional calculus:

\[
\begin{array}{c@{\qquad}c}
  % left cell: the diagram
  \begin{tikzcd}[column sep=huge,row sep=huge]
    \mathrm{Reg}_{\mathbb{F}}(\sigma(A))
      \arrow[r,hook,"\iota"]
      \arrow[dr,"\Phi^{HK}_A"']
    &
    \mathrm{Borel}_{b,\mathbb{F}}(\sigma(A))
      \arrow[d,"\Phi^{B}_A"]
    \\
    &
    \mathrm{B}(H)
  \end{tikzcd}
&
  % right cell: the equation
  \hspace{1cm} \displaystyle
  \Phi_A^{HK} \;=\; \Phi_A^B \,\circ\, \iota
\end{array}
\]

The rather interesting case for both $\mathrm{Reg}_{\mathbb{F}}(\sigma(A))$ and $\mathrm{Borel}_{b,\mathbb{F}}(\sigma(A))$ is that they are not always separable.

\begin{theorem}[$(\mathrm{Reg}_{\mathbb{F}}(\sigma(A)), \|\cdot\|_{\infty}, \mathbb{F})$ is not always Separable]
If $\sigma(A)$ contains a nondegenerate interval, $(\mathrm{Reg}_{\mathbb{F}}(\sigma(A)), \|\cdot\|_{\infty}, \mathbb{F})$ is not separable.
\end{theorem}
\begin{proof}
Consider $\epsilon =1$. It is sufficient to show that our function space does have an uncountable $1$-separated subset. Let $\{H_c\}_{c \in \sigma(A)}$ be the collection of Heaviside functions.

As for all distinct $y_1,y_2 \in \sigma(A)$, we have:
$$ d(H_{y_1}, H_{y_2}) : = \|H_{y_1} - H_{y_2}\|_{\infty} = 1$$
we conclude that $\{H_c\}_{c \in \sigma(A)}$  is an uncountable $1$-separated subset of $\mathrm{Reg}_{\mathbb{F}}(\sigma(A))$.

This proves our claim.
\end{proof}
\newpage
The same idea can be applied to $(\mathrm{Borel}_{b,\mathbb{F}}(\sigma(A)), \|\cdot\|_{\infty}, \mathbb{F})$.
\begin{theorem}[$(\mathrm{Borel}_{b,\mathbb{F}}(\sigma(A)), \|\cdot\|_{\infty}, \mathbb{F})$ is not always Separable]
If $\sigma(A)$ contains a nondegenerate interval, $(\mathrm{Borel}_{b,\mathbb{F}}(\sigma(A)), \|\cdot\|_{\infty}, \mathbb{F})$ is not separable.
\end{theorem}
\begin{proof}
Consider $\epsilon = 1$. It is sufficient to show that our function space does have an uncountable $1$-separated subset. Let $\{\mathds{1}_x\}_{x \in \sigma(A)}$ be our candidate.

As for all distinct $y_1, y_2 \in \sigma(A)$, we have:
$$d(\mathds{1}_{y_1}, \mathds{1}_{y_2}):= \|\mathds{1}_{y_1} - \mathds{1}_{y_2}\|_{\infty} = 1 $$
 we conclude that $\{\mathds{1}_x\}_{x \in \sigma(A)}$ is an uncountable $1$-separated subset of $\mathrm{Borel}_{b,\mathbb{F}}(\sigma(A))$.

This proves our claim.
\end{proof}
\bigskip 

As a consequence, we do not have a countable set that can approximate each element of our function spaces arbitrarily well. This apparent obstruction does not stop us from using the density of step functions in the space of regulated functions. For the space of bounded, Borel-measurable functions, nice approximation properties often lack. 

This will become apparent once we consider the following two examples.

\bigskip 

\begin{example}[The multiplication operator for bounded, Borel-measurable functions]  
\end{example}
Let $\sigma(A) = [0,1]$ and  $\mathrm{L^2}([0,1]) : =\mathrm{L}^2([0,1], Bor([0,1]),\lambda)$ be the Lebesgue space under the Lebesgue measure.
For both spaces, we will consider the multiplication operator on $\mathrm{L}^2([0,1])$ with respect to some $f \in \mathrm{Borel}_{b,\mathbb{F}}([0,1])$:
$$M_f:  \mathrm{L}^2([0,1])\longrightarrow \mathrm{L}^2([0,1]), \quad g \longmapsto fg$$
A useful property of the multiplication operator is that its operator norm coincides with the supremum norm of the underlying function.
\begin{proposition}For all $f\in \mathrm{Borel}_{b,\mathbb{F}}(\,[0,1]\,):\|M_f\|_{op} = \|f\|_{\infty}$. \end{proposition}
\begin{proof}
We suppose $f \in \mathrm{Borel}_{b,\mathbb{F}}([0,1])$ and $h \in \mathrm{L^2}([0,1])$. 

Then:
$$(\|M_f (h)\|_{\mathrm{L^2}})^2 = (\|fh\|_{\mathrm{L^2}})^2= \int\limits_{[0,1]}(|f(x)h(x)|_{\mathbb{F}})^2 \, d\lambda(x) \leq (\|f\|_{\infty})^2(\|h\|_{\mathrm{L^2}})^2$$

By taking the supremum norm of all unit-normed square-integrable functions, we get:
$$\|M_f\|_{op} \leq \|f\|_{\infty}$$
For the reverse inclusion, we fix some $\epsilon>0$, and consider the following set:
$$E_{\epsilon}:= \left\{x \in [0,1] \mid |f(x)|_{\mathbb{F}} > \|f\|_{\infty} - \epsilon \right\}$$
We claim that this set has positive Lebesgue measure.

Indeed, if $\lambda(E_{\epsilon}) =  0 $, we will have:
$$ |f(x)|_{\mathbb{F}} \leq \|f\|_{\infty} - \epsilon \quad \text{ for almost all } x\in[0,1]$$
This contradicts he minimality of $\|f\|_{\infty}$. 

Hence, $\lambda(E_{\epsilon}) >  0$.

Let:
$$h = \frac{\mathds{1}_{E_{\epsilon}}}{\|\mathds{1}_{E_{\epsilon}}\|_{\mathrm{L^2}}} = \frac{\mathds{1}_{E_{\epsilon}}}{\lambda(E_{\epsilon})}$$
Then, by acting on $h$ with the multiplication operator, we get:
$$(\|M_f(h)\|_{\mathrm{L^2}})^2 = (\|f h\|_{\mathrm{L^2}})^2 = \int\limits_{[0,1]}(|f(x)h(x)|_{\mathbb{F}})^2 \,d\lambda(x) \geq \frac{1}{\lambda(E_{\epsilon})}\int\limits_{E_{\epsilon}} (\|f\|_{\infty} - \epsilon)^2 \, d\lambda(x) =  (\|f\|_{\infty} - \epsilon)^2$$
Consequently:
$$\|M_f\|_{op} \geq \|M_f(h)\|_{\mathrm{L^2}} \geq \|f\|_{\infty} - \epsilon$$
As $\epsilon > 0$ was arbitrarily chosen, we get the reverse inequality:
$$\|M_f\|_{op} \geq \|f\|_{\infty}$$
This proves our claim.
\end{proof}
\bigskip

Another property we will use in our analysis is the linearity of the multiplication function in its subscript.
\begin{proposition}[Subscript-Linearity of the Multiplication Operator]

\mbox{}\\
For all $f_1,f_2 \in\mathrm{Borel}_{b,\mathbb{F}}([0,1])$ and $\alpha_1, \alpha_2 \in \mathbb{F}$,  we have:
$$M_{\alpha_1f_1 + \alpha_2 f_2} = \alpha_1M_{f_1} + \alpha_2 M_{f_2}$$  
\end{proposition}

\begin{proof} By considering $f_1,f_2 \in\mathrm{Borel}_{b,\mathbb{F}}([0,1]), \, \alpha_1, \alpha_2 \in \mathbb{F}$ and $h \in \mathrm{L^2([0,1])}$, we get:
\begin{align*}
    \left(M_{\alpha_1f_1} + M_{\alpha_2f_2}\right)(h) &:= M_{\alpha_1f_1}(h) + M_{\alpha_2 f_2}(h) \\
    &= \left(\alpha_1 f_1 + \alpha_2 f_2\right)(h) \\
    &= \left(M_{\alpha_1 f_1 + \alpha_2 f_2}\right)(h)
\end{align*}

This proves subscript-linearity.
\end{proof}
\bigskip

For the Borel functional calculus, we will choose the subscript of our multiplication operator to be the Dirichlet function $\mathds{1}_{\mathbb{Q}}$ on $[0,1]$. As this function is not regulated, it fails to be uniformly approximable by step functions. Indeed, by the density of the rationals and irrationals in the reals, we have that for each partition of mutually disjoint subsets $\{I_i\}_{i = \overline{1,n}}$ of $[0,1]$ and any step function:
$$s = \sum\limits_{i=1}^nc_i \mathds{1}_{I_i}$$ 
we get: $$\sup\limits_{\lambda \in I_i}|s(\lambda) - \mathds{1}_{\mathbb{Q}}(\lambda)|_{\mathbb{F}} = \max\{|c_i -1|_{\mathbb{F}},|c_i|_{\mathbb{F}}\} > 0 \quad \text{ for } i=\overline{1,n}$$
Equivalently:
$$\|s - \mathds{1}_{\mathbb{Q}}\|_{\infty} > 0$$
In particular, even if we may have a sequence of step functions $\{s_n\}_{n \in \mathbb{N}}$ converging pointwise to $\mathds{1}_{\mathbb{Q}}$, the sequence will fail to approximate Dirichlet's function uniformly.

By the established properties of the multiplication operator:
$$\|M_{s_n} - M_{\mathds{1}_{\mathbb{Q}}}\|_{op} = \|M_{s_n - \mathds{1}_{\mathbb{Q}}}\|_{op} = \|s_n - \mathds{1}_{\mathbb{Q}}\|_{\infty} > 0$$

Hence, not every operator can be described as the operator limit of nicer, simpler operators.

\medskip

\begin{example}[The multiplication operator for regulated functions]  
\end{example}
For the $HK$ functional calculus, we will take a function which mimics the erratic behaviour of Dirichlet's function at rational points: Thomae's function. 

On $[0,1]$, we define Thomae's function as follows:
$$
t(\lambda) : =
\begin{cases}
    \frac{1}{q},  & \text{for }\,\, \lambda = \frac{p}{q}\in \mathbb{Q} \cap [0,1] \text{ and }  (p,q) = 1\\
    0, & \text{else }
\end{cases}
$$
By constructing the following pointwise convergent sequence of $t$:
$$
\widehat{t}_n(\lambda) : =
\begin{cases}
    0,  & \text{for }\,\, x = 0\\
    1, & \text{for }\,\,x \in (0,\frac{1}{n}]\\
    f(x), &\text{for }\,\, x \in (\frac{1}{n},1]
\end{cases}
\qquad \text{for all } n \in \mathbb{N}
$$
we note that for every $n \in \mathbb{N}$ and $x \in (0,\frac{1}{n}] \cap [0,1]\setminus\mathbb{Q}$, we have:
$$|\widehat{t}_n(x) - t_n(x)|_{\mathbb{R}} = 1$$
Hence:
$$\|t_n - t||_{\infty} \geq 1$$
This leads to the same problem as in the Borel functional calculus. However, by the density of the step functions in the space of regulated functions, we have some sequence of step functions $\{t_n\}_{n \in \mathbb{N}}$ converging uniformly to Thomae's function.

Consequently:
$$\|M_{t_n} - M_t\|_{op} = \|M_{t_n - t}\|_{op} = \|t_n - t\|_{\infty} \xrightarrow[\,n \,\to\, \infty\,]{\,\|\cdot\|_{\mathbb{R}}\,} 0$$

These two examples show us how disregarding the defining property of regulated functions leads to pathologies when wanting to talk about operators as an operator norm limit of nicer operators.

%% file: 7-UHK.tex
\chapter{The Unbounded \texorpdfstring{$HK$}{TEXT} Functional Calculus} \label{ch: unboundedhk}
In this chapter, we will extend the $HK$ functional calculus to unbounded self-adjoint operators at the cost of losing the uniform approximability of regulated functions by step functions and altering the spectral mapping theorem. 

\bigskip

Given a Hilbert space $(H,\mathbb{F})$, we start off by considering an unbounded self-adjoint operator $A :D(A) \subseteq H \longrightarrow H$, where $\mathrm{D}(A)\subseteq H$ is a vector subspace. In the unbounded case, the spectrum only preserves closedness with respect to the Euclidean subspace topology of $\mathbb{R}$. This creates problems when wanting to apply gauge-fine tagged partitions to unbounded intervals.
One natural way of tackling this problem is to break down the spectrum of $A$ into compact pieces, realise the local $HK$ functional calculus on these parts, and, finally, define the unbounded functional calculus as a limit of the bounded one.

\bigskip

Since in the unbounded case the notion of an adjoint makes sense only for densely defined operators, we will need to make a suitable choice for the domain of our functional calculus.

For  $f \in \mathrm{Reg_{\mathbb{F}}}(\sigma(A))$, let:
$$\mathrm{D}(f(A)): = \left\{ x \in H \,\middle|\, \,\,\int\limits_{\sigma(A)}(|f(\lambda)|)^2 \, d\mu_{x,x}(\lambda) < \infty \right\}$$

\begin{proposition} $D(f(A))$ is dense in $H$. \end{proposition}
\begin{proof}
We fix some $x \in H$ and consider the following sequence:
$$\text{For all } n \in \mathbb{N}: x_n: = E([-n,n])x$$
By the properties of the projection valued measure, we have:
$$\|x_n - x\|_{H} = \|E([-n,n])x - x\|_H \xrightarrow[\,n \,\to\, \infty\,]{\,|\cdot|_{\mathbb{R}}\,} \|x - x\|_H = 0$$
Hence:
$$x_n \xrightarrow[\,n \,\to\, \infty\,]{\,\|\cdot\|_{\mathbb{H}}\,}x$$
Moreover, since regulated functions are bounded on compact intervals, we have:
\begin{align*}
I_{x_n,x_n}((|f|)^2) &= \int\limits_{\sigma(A)}(|f(\lambda)|_{\mathbb{F}})^2 \, d\mu_{x_n,x_n}(\lambda) \\
&= \int\limits_{[-n,n]}(|f(\lambda)|_{\mathbb{F}})^2 \, d\mu_{x_n,x_n}(\lambda)\\ 
&\leq \|\restr{f}{[-n,n]}\|_{\infty}\,\mu_{x_n,x_n}([-n,n])\\
&\leq \|\restr{f}{[-n,n]}\|_{\infty} (\|x_n\|_H)^2
\end{align*}
But this shows that :
$$E([-n,n])H \subseteq \mathrm{D}(f(A)) \,\,\text{ for each } n \in \mathbb{N}$$
Thus:
$$\bigcup\limits_{n=1}^\infty E([-n,n])H \subseteq \mathrm{D}(f(A))$$
But since this subset is dense in $H$, we get by the monotonicity of the topological closure that $\mathrm{D}(f(A))$ is also dense in $H$.

This proves our claim. 
\end{proof}
\medskip 
It is precisely because of this step that we can talk about self-adjoint unbounded operators. We note that by considering the identity function on $\mathrm{Reg}_{\mathbb{F}}(\sigma(A))$, we get the domain of $A$:
$$\mathrm{D}(A): = \left\{ x \in H \,\middle|\, \int_{\sigma(A)}\lambda^2 \, d\mu_{x,x}(\lambda) < \infty \right\}$$
\bigskip

We can now start with our construction.

Let $\sigma(A)$ be the spectrum of $\left(A, \mathrm{D}(A)\right)$ on some Hilbert space $(H, \mathbb{F})$.

Consider its truncation $\sigma_{[-n,n]}(A) := \sigma(A) \cap [-n,n]$. 

Then, by the bounded $HK$ functional calculus, we get:
$$\text{For all } f \in \mathrm{Reg_{\mathbb{F}}}(\sigma(A))  \text{ and } x,y \in H:  I^n_{x,y}(f) : = \int\limits_{\sigma_n(A)}f(\lambda) \, d\mu_{x,y}(\lambda)$$

\begin{proposition}$\{I^n_{x,y}(f)\}_{n \in \mathbb{N}}$ is Cauchy. \end{proposition}
\begin{proof}
Let $x \in \mathrm{D}(f(A)), \, y \in H$ and  $f \in \mathrm{Reg_{\mathbb{F}}}(\sigma(A))$.

For brevity, we will consider the following notations: 
\begin{align*}
    \widehat\sigma_{n_1,n_2}(A)&:= \sigma(A) \cap ([-n_1,n_1]\setminus[-n_2,n_2])\\
    \widetilde{\sigma}_n(A) &:= \sigma(A) \setminus[-n_1,n_1]\\
    \widehat{I}^{n_1,n_2}_{x,x}(f)&:= \int\limits_{\widehat\sigma_{n_1,n_2}(A)}f(\lambda)\,d\mu_{x,x}(\lambda)\\
    \widetilde{I}^n_{x,x}(f)&:=\int_{\widetilde{\sigma}_n(A)}f(\lambda)\,d\mu_{x,x}(\lambda)
\end{align*}
By considering $\epsilon > 0$ and $n_1> n_2  >  N$ for some $N \in \mathbb{N}$, we have:
\begin{align*}
    \left(\left|I^{n_1}_{x,y}(f) - I^{n_2}_{x,y}(f)\right|_{\mathbb{F}}\right)^2
    &\leq \left(\left|\int\limits_{\,\,\,\widehat\sigma_{n_1,n_2}(A)}f(\lambda) \,d\mu_{x,y}(\lambda)\right|_{\mathbb{F}}\right)^2\\
    &\leq \lim_{m \to \infty}\left(\sum\limits_{j=1}^m|f(t_j)|_{\mathbb{F}} \, \left|\mu_{x,y}(I_j)\right|_{\mathbb{F}}\right)^2\\
    &\leq \lim_{m \to \infty}\left(\sum\limits_{j=1}^m|f(t_j)|_{\mathbb{F}} \, \mu_{x,x}(I_j)^{\frac{1}{2}}\, \mu_{y,y}(I_j)^{\frac{1}{2}} \right)^2\\
    &\leq \lim_{m \to \infty}\left(\sum\limits_{j=1}^m\left(|f(t_j)|_{\mathbb{F}}\right)^2 \, \mu_{x,x}(I_j) \right)\left(\sum\limits_{k=1}^m \mu_{y,y}(I_k)\right)\\
    &\leq \lim_{m \to \infty}\left(\sum\limits_{j=1}^m\left(|f(t_j)|_{\mathbb{F}}\right)^2 \, \mu_{x,x}(I_j) \right) \left(\|y\|_H\right)^2\\
    &= \widehat{I}^{n_1,n_2}_{x,x}((|f|)^2) \left(\|y\|_H\right)^2\\
    &\leq \widetilde{I}^N_{x,x}((|f|)^2) \left(\|y\|_H\right)^2
\end{align*}

Since $I_{x,x}((|f|)^2): = \lim \limits_{n \to \infty} I^n_{x,x}((|f|)^2)$ is finite, we can let $N$ be sufficiently large such that:
$$\widetilde{I}^N_{x,x}((|f|)^2) < \frac{\epsilon}{\|y\|^2}$$
\newpage
We disregard the case in which $y = 0$, as the underlying sequence will be the constant zero sequence, implying it is both Cauchy and convergent.

This bound shows that $\{I^n_{x,y}(f)\}$ is Cauchy in $(\mathbb{F}, | \cdot|_{\mathbb{F}}, \mathbb{F})$.

By completeness of this space, we have a unique limit:
$$I_{x,y}(f): = \lim\limits_{n  \to \infty}I^n_{x,y}(f) \,\,\text{ for all } x \in \mathrm{D}(f(A)), \,y \in H \text{ and }  f\in \mathrm{Reg}_{\mathbb{F}}(\sigma(A))$$

This completes the proof.
\end{proof}

\bigskip

We may now use Riesz's Representation Theorem in order to prove the existence and uniqueness of $f(A)$.
\begin{theorem}[The construction of $f(A)$]
Given a Hilbert Space $(H,\mathbb{F})$, and a regulated function $f \in \mathrm{Reg}_{\mathbb{F}}(\sigma(A)), \,\,f(A) $ is completely determined by the following property:
$$\text{For all } x \in \mathrm{D}(f(A)) \text{ and } y \in H: \inner{f(A)x}{y}_H = I_{x,y}(f)$$
\end{theorem}
\begin{proof}
Let $f \in  \mathrm{Reg}_{\mathbb{F}}(\sigma(A))$ and $ x\in\mathrm{D}(f(A))$. We may now consider the following map:
$$
  g_x\colon
  H
  \;\longrightarrow\;
  \mathbb{C},
  \quad
  y\;\longmapsto\;\overline{I_{x,y}(f)} 
$$
We aim to show that this map is linear and bounded.
Since linearity is a direct consequence of the linearity of the inner product, we will prove its boundedness.

\textbf{Boundedness:}

For $y \in H$, we have: 
\begin{align*}
    |g_x(y)|_{\mathbb{F}} 
    &= \left|\overline{I_{x,y}(f)}\right|_{\mathbb{F}}\\
    &= \left|{I_{x,y}(f)}\right|_{\mathbb{F}}\\
    &\leq \lim\limits_{m \to \infty} \sum\limits_{i=1}^m |f(t_i)|_{\mathbb{F}} |\mu_{x,y}(I_i)|_{\mathbb{F}}\\
    &\leq \lim\limits_{m \to \infty} \sum\limits_{i=1}^m |f(t_i)|_{\mathbb{F}}\,\mu_{x,x}(I_i)^{\frac{1}{2}}\, \mu_{y,y}(I_i)^{\frac{1}{2}}\\
    &\leq \lim\limits_{m \to \infty} \left(\sum\limits_{i=1}^m (|f(t_i)|_{\mathbb{F}})^2\,\mu_{x,x}(I_i)\right)^{\frac{1}{2}}\left(\sum\limits_{i=1}^m \mu_{y,y}(I_i)\right)^{\frac{1}{2}}\\
    &= \sqrt{I_{x,x}((|f|)^2)}\,\, \|y\|_{H}
\end{align*}

By picking $M_x = \sqrt{I_{x,x}((|f|)^2)}$, we have that:
$$ |g_x(y)|_{\mathbb{F}}  \leq M_x \|y\|_{H} \, \,\text{ for all } y \in H$$
This proves boundedness.

By Riesz's Representation Theorem, we get:
$$g_x = \inner{\cdot}{\,\,f(A)x}_H \,\,\text{ for some unique } f(A)x \in H$$ 
By the antilinearity of the inner product in the second argument, we get:
$$\inner{f(A)x}{y}_H = \overline{\inner{y}{f(A)x}_H} = I_{x,y}(f)$$
Since $x \in \mathrm{D}(f(A))$ was arbitrarily chosen, we have constructed $f(A)$ as a linear operator satisfying:
$$\inner{f(A)x}{y}_H = I_{x,y}(f) \,\,\text{ for all } x \in \mathrm{D}(f(A)) \text{ and } y \in H$$
This proves our claim.
\end{proof}

\bigskip

It remains to show that the assignment $f \longmapsto f(A)$ defines a functional calculus for $A$.

Firstly, we will prove a lemma which will be of later use:
\begin{lemma}[$f(A)$ commutes with the  Projection\hspace{0.05cm}-\hspace{-0.1cm}Valued Measure]
Given a Hilbert space $(H, \mathbb{F})$, a regulated function $ f \in \mathrm{Reg}_{\mathbb{F}}(\sigma(A))$, and the (unique) projection-valued measure $E$ from the unbounded Borel functional calculus, we have:
$$f(A)E(J) = E(J)f(A) \,\,\text{ for all } J \in Bor(\sigma(A))$$
\end{lemma}
\begin{proof}

Let $x \in \mathrm{D}(f(A)), \, y\in H$ and $J \in Bore(\sigma(A))$. 

Then:
\begin{align*}
    \inner{f(A)E(J)x}{y}_H &= \int\limits_{\sigma(A)} f(\lambda) \, d\mu_{E(J)x,y}(\lambda)\\
    &= \int\limits_{\sigma(A)}f(\lambda) \, \mathds{1}_{\hspace{-0.15cm}J}(\lambda)\, d\mu_{x,y}(\lambda)\\
    &= \int\limits_{\sigma(A)}f(\lambda) \, d\mu_{x,E(J)y}\\
    &=\inner{f(A)x}{E(J)y}_H\\
    &=\inner{E(J)f(A)x}{y}_H
\end{align*}
This proves our lemma.
\end{proof}

\newpage 
\begin{theorem}[The Unbounded $HK$ Functional Calculus is a bona fide Functional Calculus]
Given the $*$-algebra of regulated functions $(\mathrm{Reg}_{\mathbb{F}}(\sigma(A)),\mathbb{F})$, the following map is a $*$-algebra homomorphism:
\[
  \Phi^{HK}_{A}\colon
  \mathrm{Reg}_{\mathbb{F}}(\sigma(A))
  \;\longrightarrow\;
  \{f(A)\mid f \in \mathrm{Reg}_{\mathbb{F}}(\sigma(A))\},
  \quad
  f\;\longmapsto\;f(A) 
\]
\end{theorem}
\begin{proof}
For the sake of brevity, let $\mathcal{A}: = \{f(A)\mid f \in \mathrm{Reg}_{\mathbb{F}}(\sigma(A))\}$.

We will verify the following axioms in the prescribed order, noting that equalities make sense only on common domains:

For all $f_1,f_2,f_3 \in \mathrm{Reg}_{\mathbb{F}}(\sigma(A))$ and $\alpha_1,\alpha_2 \in \mathbb{F}$:
\begin{enumerate}[label=\arabic*)]
\item $\Phi^{HK}_{A}(\alpha_1 f_1 + \alpha_2f_2) = \alpha_1\Phi^{HK}_{A}(f_1) + \alpha_2 \Phi^{HK}_{A}(f_2)$
\item $\Phi^{HK}_{A}(f_1f_2) = \Phi^{HK}_{A}(f_1)\Phi^{HK}_{A}(f_2)$
\item $\Phi^{HK}_{A}(\overline{f_3}) = (\Phi_{A}^{HK}(f_3))^*$
\item $\Phi^{HK}_{A}(1_{\mathrm{Reg}_{\mathbb{F}}(\sigma(A))}) = 1_{\mathcal{A}}$
\end{enumerate}

Throughout all the verifications, we will fix $f_1,f_2,f_3 \in \mathrm{Reg}_{\mathbb{F}}(\sigma(A))$ and $\alpha_1,\alpha_2 \in \mathbb{F}$.

For $1)$, we will first establish the domain on which the equality holds.

Noting that 
$$\mathrm{D}(\alpha_1 f_1(A) + \alpha_2 f_2(A)) = \mathrm{D}(f_1(A)) \cap \mathrm{D}(f_2(A))$$
we have:
$$\text{For all } x \in \mathrm{D}(f_1(A)) \cap \mathrm{D}(f_2(A)): I_{x,x}((|f_1|)^2), \,\, I_{x,x}((|f_2|)^2) < \infty$$
By the following pointwise inequality:
\begin{align*}
  (|\alpha_1f_1(\lambda) +\alpha_2f_2(\lambda)|_{\mathbb{F}})^2 
  &\leq (|\alpha_1f_1(\lambda)|_{\mathbb{F}} + |\alpha_2f_2(\lambda)|_{\mathbb{F}})^2 \\
  &= (|\alpha_1f_1(\lambda)|_{\mathbb{F}})^2 + 2|\alpha_1\alpha_2|_{\mathbb{F}}|f_1(\lambda)f_2(\lambda)|_{\mathbb{F}} + (|\alpha_2f_2(\lambda)|_{\mathbb{F}})^2\\
  &\leq 2((|\alpha_1f_1(\lambda)|_{\mathbb{F}})^2 + (|\alpha_2f_2(\lambda)|_{\mathbb{F}})^2)
\end{align*}
 we get:
$$ \hspace{-0.5cm}I_{x,x}((|\alpha_1f_1 + \alpha_2 f_2|)^2) \leq I_{x,x}((|\alpha_1f_1|)^2) + I_{x,x}((|\alpha_1f_1|)^2) = |\alpha_1|_{\mathbb{F}}\, I_{x,x}((|f_1|)^2) +  |\alpha_2|_{\mathbb{F}}\, I_{x,x}((|f_2|)^2 < \infty$$
 This proves the following subset relationship:
 $$\mathrm{D}(\alpha_1 f_1(A) + \alpha_2 f_2(A)) \subseteq \mathrm{D}((\alpha_1 f_1 + \alpha_2 f_2)(A))$$
 Moreover, since:
 $$\text{For all } f \in \mathrm{Reg}_{\mathbb{F}}(\sigma(A)): \bigcup\limits_{n =1}^\infty E([-n,n])H \subseteq \mathrm{D}(f(A))$$
 it implies that $\mathrm{D}(\alpha_1 f_1(A) + \alpha_2f_2(A))$ is dense in  $H$.

Let $ x \in \mathrm{D}(\alpha_1 f_1(A) + \alpha_2 f_2(A))$ and $y \in H$. By choosing this subset as the common domain in which we will verify linearity, we have:
\begin{align*}
    \inner{(\alpha_1 f_1 + \alpha_2 f_2)(A)x}{y}_H 
    &=I_{x,y}(\alpha_1f_1 + \alpha_2f_2)\\
    &=\alpha_1 I_{x,y}(f_1) + \alpha_2I_{x,y}(f_2) \\
    &= \alpha_1\inner{f(A)x}{y}_H + \alpha_2 \inner{f(A)x}{y}_H\\
    &=\inner{\alpha_1 f_1(A)x}{y}_H + \inner{\alpha_2f_2(A)x}{y}_H\\
    &=\inner{(\alpha_1 f_1(A) + \alpha_2 f_2(A))x}{y}_H
\end{align*}

This implies that:
$(\alpha_1 f_1 + \alpha_2 f_2)(A) = \alpha_1 f_1(A) + \alpha_2 f_2(A) \,\,\text{ on }\, \,\mathrm{D}(\alpha_1 f_1(A) + \alpha_2 f_2(A))$
This verifies the first axiom.

\bigskip

As verifying $2)$ relies on $3)$, we will start with the latter.

For $3)$, we let $x,y \in \mathrm{D}\big(f_3(A)\big) = \mathrm{D}\big(\overline{f_3}(A)\big)$.

Then:

\begin{align*}
  \inner{\overline{f_3}(A)x}{y}_H 
  &= \int\limits_{\sigma(A)} \overline{f_3(\lambda)} \, d\mu_{x,y}(\lambda)\\
  &= \int\limits_{\sigma(A)} \overline{f_3(\lambda) \, d\mu_{y,x}(\lambda)}\\
  &= \overline{\int\limits_{\sigma(A)} f_3(\lambda) \, d\mu_{y,x}(\lambda)}\\
  &= \overline{\inner{f_3(A)x}{y}_H}\\
  &= \inner{x}{f_3(A)^{*}y}_H
\end{align*}

Hence: $$\overline{f_3}(A) = f_3(A)^{*}$$
This verifies the third axiom.

For $2)$, we will test the equality on $$\mathrm{D}(f_1(A)f_2(A)): =\left\{x \in \mathrm{D}(f_2(A)) \mid f_2(A)x \in \mathrm{D}(f_1(A))\right\}$$
Firstly, we will prove that this set is a dense subset of $\mathrm{D}((f_1f_2)(A))$.

By considering some $x \in \mathrm{D}(f_1(A)f_2(A))$ and $y \in H$, we have, together with the commutation of $f_2(A)$ with $E$, that:
\begin{align*}
    I_{f_2(A)x, f_2(A)x}((|f_1|)^2) 
    &= \int\limits_{\sigma(A)}(|f_1(\lambda)|_{\mathbb{F}})^2 \, d\mu _{f_2(A)x, f_2(A)x}\\
    &= \int\limits_{\sigma(A)}(|f_1(\lambda)|_{\mathbb{F}} |f_2(\lambda)|_{\mathbb{F}})^2 \, d\mu _{x, x}\\
    &=I_{x,x}((|f_1f_2|)^2) < \infty
\end{align*}
This implies that:
$$\mathrm{D}(f_1(A)f_2(A)) \subseteq \mathrm{D}((f_1f_2)(A))$$

Moreover, since:
 $$\text{For all } f \in \mathrm{Reg}_{\mathbb{F}}(\sigma(A)): \bigcup\limits_{n =1}^\infty E([-n,n])H \subseteq \mathrm{D}(f(A))$$
 we get:
 $$ \bigcup\limits_{n =1}^\infty E([-n,n])H \subseteq \mathrm{D}(f_2(A))$$
 As a consequence of the commutation of $f_2(A)$ with $E$, we derive the following fact:
 $$\text{For all } x \in \bigcup\limits_{n=1}^\infty E([-n,n])H: f_2(A)x \in \mathrm{D}(f_1(A))$$
 This shows that:
 $$\bigcup\limits_{n =1}^\infty E([-n,n])H \subseteq \mathrm{D}((f_1f_2)(A))$$
Our chosen subset is thus dense in $H$.

\medskip
We can now shift our attention to proving the second property for this chosen subset.
We consider the following sequences:
$$\text{For all }  n \in \mathbb{N}: \,\,\,\,\,f_{1,n}: = f_1\, \mathds{1}_{[-n,n] \,\cap\, \sigma(A)} \qquad f_{2,n}: = f_1\, \mathds{1}_{[-n,n] \,\cap\, \sigma(A)}$$
Since these sequences of regulated functions lie in a compact set, they are bounded. 
\newpage
Additionally, we can express them as operators through the bounded $HK$ functional calculus. For each $n \in \mathbb{N}$, we consider the sequence of step functions that uniformly approximate their subscripted regulated functions:
$\left\{s_k^{f_{1,n}}\right\}_{k \in \mathbb{N}},\left\{s_k^{f_{2,n}}\right\}_{k \in \mathbb{N}}$.

By the bounded $HK$ functional calculus, we have:
$$\displaystyle \inner{(s_k^{f_{1,n}}s_l^{f_{2,n}})(A)x}{y}_H = \inner{s_k^{f_{1,n}}(A)s_l^{f_{2,n}}(A)x}{y}_H \quad \text{ for all } k,l,n \in \mathbb{N}$$
As we let $k$ and $l$ tend to infinity, we recover the bounded $HK$ functional calculus  for the sequences of regulated functions:
$$\inner{(f_{1,n} f_{2,n})(A)x}{y}_H = \inner{f_{1,n}(A)f_{2,n}(A)x}{y}_H \quad \text{ for all } n \in \mathbb{N}$$
By letting $n$ grow arbitrarily large, we get:
$$\inner{(f_1f_2)(A)x}{y}_H = \inner{f_1(A)f_2(A)x}{y}_H$$
This verifies the second axiom.

\bigskip 

For $4)$, the domain $\mathrm{D}(1_{\mathrm{Reg}_{\mathbb{F}}(\sigma(A))}(A))$ is trivially dense in $H$:
$$
\mathrm{D}(1_{\mathrm{Reg}_{\mathbb{F}}(\sigma(A))}(A))=\Bigg\{\,x\in H \;\Bigg|\;\int_{\sigma(A)}d\mu_{x,x}(\lambda)<\infty\Bigg\}=\Bigg\{\,x\in H \;\Bigg|\;\|x\|_H^2<\infty\Bigg\}
$$
Hence:
$$\text{For all } x,y \in H: \inner{1_{\mathrm{Reg}_{\mathbb{F}}(\sigma(A))}(A) x}{y}_H = \int\limits_{\sigma(A)}d\mu_{x,y}(\lambda) = \inner{x}{y}_H$$
This proves that:
$$1_{\mathrm{Reg}_{\mathbb{F}}(\sigma(A))}(A) = 1_{\mathcal{A}}$$
With this final remark, we have verified the last axiom.

This completes our proof.
\end{proof}

%% file: 7.1-SP.tex
\chapter{The Spectral Mapping Theorem for Regulated Functions} \label{ch: spectralmappingtheorem}
In this chapter, we will leverage the properties of regulated functions to provide an informative description of $\sigma(f(A))$. Afterwards, we will give a few examples outlining the advantages of the $HK$ functional calculus.
\bigskip

From the continuous functional calculus, we have the following equality:
$$\sigma(f(A)) = f(\sigma(A)) \quad \text{for all } f \in \mathrm{C}_{\mathbb{F}}(\sigma(A))$$

Unfortunately, this equality breaks in case we extend $\mathrm{C}_{\mathbb{F}}(\sigma(A))$ to $\mathrm{Reg}_{\mathbb{F}}(\sigma(A))$. Considering the multiplication operator $M_x$ in the Hilbert space $(\mathrm{L^2}([a,b], Bor([a,b]), \lambda), \inner{\cdot}{\cdot}_{L^2},\mathbb{R})$, we fix some $c \in [a,b]$. Since singletons are Lebesgue-negligible, we will have that $\mathds{1}_{c}(A) = 0$. As a result $\sigma(\mathds{1}_{c}(A)) = {0}$, while $\mathds{1}_{c}(\sigma(A)) = \{0,1\}$. This pathology suggests the need for a new, easily computable formula for $\sigma(f(A))$.

\bigskip

Firstly, we will need some intermediate results:
\begin{proposition}[$Ran(E(\{\lambda\})) = \mathrm{Ker}(A - \lambda I)$] Given a Hilbert space $(H, \mathbb{F})$, a self-adjoint operator $(A, \mathrm{D}(A))$, and the (unique) projection-valued $E$ measure from the Borel functional calculus, we have:
$$\text{For all } \lambda \in \mathbb{R}: \mathrm{Ran}(E(\{\lambda\}) = \mathrm{Ker}(A - \lambda I)$$
\end{proposition}
\begin{proof}
    Let $\lambda \in \mathbb{R}$. For $\mathrm{Ran}(E(\{\lambda\}) \subseteq  \mathrm{Ker}(A - \lambda I)$, we remark the following pointwise equality:
    $$\text{For all } t \in \mathbb{R}: (t - \lambda)\mathds{1}_{\{\lambda\}}(t) = 0$$
    By lifting the expression to the $HK$ functional calculus, we get: $(A- \lambda I)E(\{\lambda\}) = 0$
    This proves the forward direction.

    For $\mathrm{Ker}(A - \lambda I) \subseteq \mathrm{Ran}(E(\{\lambda\})$, we fix $x \in \mathrm{Ker}(A - \lambda I)$. Then, from:
    $$0 = \left(\|(A - \lambda I)x\|_H\right)^2 = \int\limits_{\sigma(A)} (\left|t - \lambda|_{\mathbb{R}}\right)^2 \, d\mu_{x,x}(t)$$
    it follows that $\mu_{x,x}(\sigma(A)\setminus\{\lambda\}) = 0$.\\
    Furthermore:
    \begin{align*}
       (\|x - E(\{\lambda\})x\|_H)^2 
       &=\inner{(I - E(\{\lambda\}))x}{(I - E(\{\lambda\}))x}_H \\
       &=\inner{(I - E(\{\lambda\}))x}{x}_H \\
       &=\mu_{x,x}\left(\sigma(A) \setminus\{\lambda\}\right) \\
       & = 0
    \end{align*}
This implies that $x \in Ran(E(\{\lambda\}))$, finishing our proof.
\end{proof}
While, in and of itself, the previous proposition will not be used further on, it yields an interesting corollary regarding the point spectrum of a self-adjoint operator.
\begin{corollary}
    Given a Hilbert space $(H, \mathbb{F})$, a self-adjoint operator $(A, \mathrm{D}(A))$, and the (unique) projection-valued $E$ measure from the Borel functional calculus, we have:
    $$\sigma_p(A):= \{\lambda \in \mathbb{R} \mid E(\{\lambda\}) \neq 0\}$$
\end{corollary}
\begin{proof}
    $$\lambda \in \sigma_p(A) \text{ if and only if } \mathrm{Ker}(A - \lambda I) \neq 0 \text{ if and only if } E(\{\lambda\}) \neq 0$$
    Hence: 
     $$\sigma_p(A):= \{\lambda \in \mathbb{R} \mid E(\{\lambda\}) \neq 0\}$$
\end{proof}
We can now prove the principal theorem of this chapter.
\begin{theorem}[The Spectral Theorem for Regulated Functions]
     Given a Hilbert space $(H, \mathbb{F})$, a self-adjoint operator $(A, \mathrm{D}(A))$, the (unique) projection-valued $E$ measure from the Borel functional calculus, and a regulated function $f \in \mathrm{Reg}_{\mathbb{F}}(\sigma(A))$, $\sigma(f(A))$ can be described as follows:
     $$\sigma(f(A)) = \overline{f(\sigma_p(A)) \cup \left\{f(c^-), f(c^+) \mid  c\in \sigma(A) \setminus \sigma_p(A)\right\}}$$
\end{theorem}
\newpage
\begin{proof}
    For brevity, let: 
    $$\Sigma(f): = f(\sigma_p(A)) \cup \left\{f(c^-), f(c^+) \mid  c\in \sigma(A) \setminus \sigma_p(A)\right\}$$
    For $\overline{\Sigma(f)} \subseteq \sigma(f(A))$, we have two cases to verify:
    \begin{enumerate}
        \item $f(\sigma_p(A)) \subseteq \sigma(f(A))$
        \item $\Sigma(f)\setminus f(\sigma_p(A)) \subseteq \sigma(f(A)) $
    \end{enumerate}
    For $1)$, we fix $\epsilon > 0$ and $\lambda \in \sigma_p(A)$. Then, as $\lambda \in f^{-1}(B(f(\lambda), \epsilon))$, by the monotonicity of the spectral measure, we have:
    $$E(f^{-1}(B(f(\lambda), \epsilon)) \geq E(\{\lambda\}) \neq 0$$
    Recalling that: 
    $$\sigma(f(A)) = \left\{\lambda \in \mathbb{C} \mid \text{for all } \epsilon > 0: E(f^{-1}(B(\lambda,\epsilon))  \neq 
 0\right\}$$
    we conclude:
    $$f(\lambda) \in \sigma(f(A))$$
    For $2)$, we fix $\delta > 0$, assuming, without loss of generality, that $E((c - \delta, c)) \neq 0$. By the regularity of $f$, we get:
    $$E(f^{-1}(B(f(c^-),\epsilon)) \geq E((c- \delta, x)) \neq 0$$
    This shows that: $$f(c^{-})\in \sigma(f(A))$$
    By interchanging left neighbourhoods $(c - \delta, c)$ with right ones $(c, c +\delta)$, we obtain: $$f(c^+) \in \sigma(f(A))$$ 
    Since the spectrum of a normal operator is closed, the forward inclusion is proven.\\
    For the reverse inclusion, we fix $n \in \mathbb{N}$ and cover $\sigma(A) \setminus \mathrm{Disc}_{\mathbb{F}}(f, \sigma(A))$ by a countable family $\mathcal{I_n}$ of open intervals in $\mathbb{R}$ such that:
    $$\sigma(A) = \bigsqcup\limits_{I_n \in \mathcal{I_n}} (I_n \cap\sigma(A)) \sqcup \mathrm{Disc}_{\mathbb{F}}(f, \sigma(A))$$ 
    $$\quad \sup_{x,y \,\in\, I_n \,\cap \,\sigma(A)}|f(x) - f(y)|_{\mathbb{F}} < \frac{1}{n} \quad \text{for all } I_n \in \mathcal{I_n}$$
    Letting $\lambda \in \sigma(f(A))$ implies $E\left(f^{-1}\left(B\left(\lambda, \frac{1}{n}\right)\right)\right) \neq 0$.

    From: $$f^{-1}\left(B\left(\lambda, \frac{1}{n}\right)\right) = \bigsqcup\limits_{I_n \in \mathcal{I_n}}\left(f^{-1}\left(B\left(\lambda, \frac{1}{n}\right)\right) \cap I_n \cap \sigma(A)\right) \sqcup \left(f^{-1}\left(B\left(\lambda, \frac{1}{n}\right)\right) \cap \mathrm{Disc}_{\mathbb{F}}(f, \sigma(A))\right)$$
    and the properties of the projection-valued measure, we have that at least one of the following holds:
    \begin{enumerate}
        \item $E(f^{-1}(B(\lambda,\frac{1}{n})) \cap \{\lambda\}) \neq 0$
        \item $E(f^{-1}(B(\lambda, \frac{1}{n}) \cap I_n) \neq 0$
    \end{enumerate}
    For the first case, we get:
    $$\text{For all } n \in \mathbb{N}, \text{ there exists } \alpha_n \in \sigma_p(A) \text{ such that: } |f(\alpha_n) - \lambda|_{\mathbb{F}} < \frac{1}{n}$$
    This implies that $\lambda \in \overline{f(\sigma_p(A)})$.

    For the second case, we pick a non-zero element $x \in Ran(E(f^{-1}(B(\lambda, \frac{1}{n})) \cap I_n)$. Then $\mu_{x,x}(f^{-1}(B(\lambda, \frac{1}{n})) \cap I_n) > 0$. Letting $S_n = supp(\mu_{x,x}) \cap f^{-1}(B(\lambda, \frac{1}{n})) \cap I_n$, we set $c_n \in \{\inf S_n, \sup S_n\}$. Without loss of generality, we assume $c_n = \inf S_n$. Then:
    $$\text{For all } \delta_1 > 0, \text{there exists } s \in S, \text{ such that: } c_n < s < c_n + \delta_1$$
    By choosing $\eta > 0$ small enough, such that:
    $$(s - \eta, s + \eta) \subseteq (c_n, c_n + \delta) \cap I_n$$
    we arrive at the following result:
    $$\mu_{x,x}\left(\left(c_n, c_n + \delta_1\right) \cap f^{-1}\left(B\left(\lambda, \frac{1}{n} \right)\right) \cap I_n\right) > \mu_{x,x}\left(\left(s - \eta , s + \eta\right) \cap f^{-1}\left(B\left(\lambda, \frac{1}{n} \right)\right) \cap I_n\right) > 0$$
    This proves that:
    $$\text{For all } \delta_1 > 0: E\left(\left(c_n, c_n + \delta_1\right) \cap f^{-1}\left(B\left(\lambda, \frac{1}{n} \right)\right)\right) \neq 0$$
    By the regularity of $f$:
    $$\text{For all } n \in \mathbb{N}, \text{ there exists } \delta_2 > 0, \text{ such that for all } t \in (c_n, c_n+\delta_2) \, \cap \, \sigma(A): |f(c_n^+) - f(t)|_{\mathbb{F}} < \frac{1}{n}$$
    Taking $\delta_3: = \min\{\delta_1, \delta_2\}$, gives us some $t \in (c_n, c_n +\delta_3)  \cap f^{-1}\left(B\left(\lambda, \frac{1}{n} \right)\right)$ such that: 
    $$|f(c_n^+) - \lambda|_{\mathbb{F}} \leq |f(c_n^+) - f(t)|_{\mathbb{F}} + |f(t) - \lambda |_{\mathbb{F}} < \frac{2}{n}$$
    This proves that $\lambda \in \overline{\Sigma(f) \setminus f(\sigma_p(A))}$. The case when $c_n = \sup S_n$ is treated similarly.
    This finishes our proof.
\end{proof}

Having established sufficiently many aspects of our functional calculus, we will proceed with some examples.

\begin{example}[Thomae's Function] \end{example}
Given a self-adjoint operator $A \in \mathrm{B}(H)$ with the spectrum $\sigma(A) \subsetneq (0,1)$, we define Thomae's function on $\sigma(A)$ as:
$$ t(\lambda) :=
\begin{cases}
    \frac{1}{q},  & \text{for }\,\, \lambda = \frac{p}{q}\in \mathbb{Q} \cap \sigma(A) \text{ and }  (p,q) = 1\\
    0, & \text{else }
\end{cases}$$

\begin{figure}[H]
  \centering
  \includegraphics[width=6.5\textwidth,height=6.5cm,keepaspectratio]{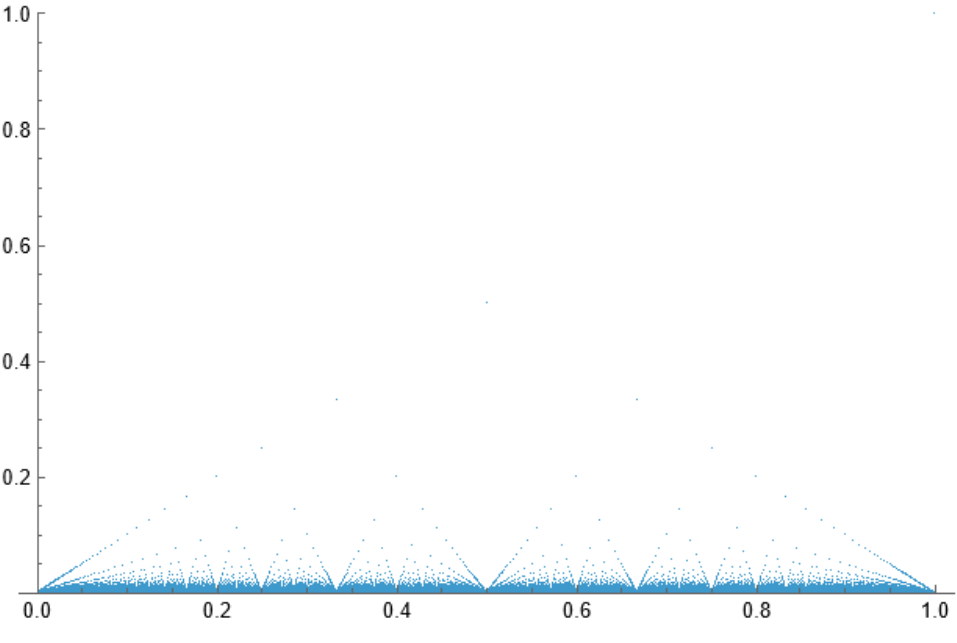}
  \caption{Tomae's Function on $(0,1)$}
\end{figure}
It is a classic result one encounters in a real-analysis course that the discontinuities of Thomae's function are the rational points. 

This, however, does not stop us from evaluating it under the $HK$ functional calculus.
Because it has countably infinite discontinuities, we cannot interpret it as a step function. However, we might want to approach it uniformly by a sequence of step functions.

We fix some $n \in \mathbb{N}$ and let:
$$ I _n : = \left\{\frac{p}{q} \in \mathbb{Q} \cap \sigma(A) \,\middle|\,1 \leq q \leq N, (p,q) =1\right\}$$
Consider:
$$t_n(\lambda) : = 
\begin{cases}
\frac{1}{q} ,  & \text{for }\,\, \lambda \in I_n\\
0,  & \text{else}
\end{cases}$$
As each $I_n$ is finite, we will discretize it in an increasing order as such:
$$I_n = \left\{\frac{p_i}{q_i}\right\}_{i=\overline{1,M_n}}$$
Consequently, each $t_n$ can be seen as a step function, whose entire sequence converges uniformly to $t$. 
$$ t_n = \sum\limits_{i=1}^{M_n} \frac{1}{q_i} \, \mathlarger{\mathbb{{1}}}_{\left\{\frac{p_i}{q_i}\right\}} \quad\text{ and }\quad t_n\xrightarrow[\;n\,\to\,\infty\;]{\|\cdot\|_{\infty}} t$$
Moreover, as:
$$t_n(A) = \sum\limits_{i=1}^{M_n}\frac{1}{q_i}E\left(\left\{\frac{p_i}{q_i}\right\}\right)$$
we get, by Lipschitz continuity of the $HK$ functional calculus, that:
$$t(A) = \lim\limits_{n \to \infty}\sum\limits_{i=1}^{M_n}\frac{1}{q_i}E\left(\left\{\frac{p_i}{q_i}\right\}\right) = \sum \limits_{i=1}^{\infty}\sum\limits_{\frac{p}{q}\in I_i}\frac{1}{q}E\left(\left\{\frac{p}{q}\right\}\right)  \quad \text{ in } \,\,\,\|\cdot\|_{op}$$
We may now look at the spectrum of this operator with the help of the spectral mapping theorem. Since Thomae's function has zero-valued one-sided limits, our spectrum becomes:
$$\sigma(t(A))  = \{0\} \cup \left\{\frac{1}{q} \, \middle| \, \exists \, p \in \mathbb{N}: \frac{p}{q} \in \mathbb{Q} \cap \sigma_p(A)\,,(p,q) =1\right\}$$

\begin{example}[The Position Operator]\end{example}
The following operator is of great importance in Quantum Mechanics as it gives us information on the areas in which a wavefunction is concentrated. In order to do that, we pointwisely localise a particle's configuration as such:
\begin{definition}[Position Operator]
    Given the Hilbert space $(\mathrm{L^2}(\mathbb{R}, Bor(\mathbb{R}), \lambda), \inner{\cdot}{\cdot}_{L^2},\mathbb{R})$, the position operator is defined as the following map:
    $$Q : \mathrm{D}(Q)\subseteq (\mathrm{L^2}(\mathbb{R}, Bor(\mathbb{R}), \lambda) \longrightarrow (\mathrm{L^2}(\mathbb{R}, 
    Bor(\mathbb{R}), \lambda), \quad \psi(x) \longmapsto x\psi(x) \quad \text{for all } x \in \mathbb{R}$$
    where:
    $$\mathrm{D}(Q): = \left\{\psi \in (\mathrm{L^2}(\mathbb{R}, Bor(\mathbb{R}), \lambda) \middle|\, \int\limits_{\mathbb{R}}(|x \psi(x)|_{\mathbb{R}})^2d\lambda(x) < \infty  \right\}$$
\end{definition}
In case we would like to weigh the spatial coordinates in a particular manner, we can use the $HK$ functional calculus.
The use of the calculus is justified as one easily verifies that $Q$ is an unbounded, densely defined, self-adjoint operator with $\sigma(Q) = \mathbb{R}$. \\ 
 By the unbounded $HK$ functional calculus, we have, in particular:
 $$\text{For all } f \in \mathrm{Reg}_{\mathbb{R}}(\mathbb{R}): f(Q) = \int\limits_{\mathbb{R}}f(x)\, dE(x)$$
As the result is rather dull and anticipated, we wish for another fruitful representation of the functional calculus. 
As the projection-valued measures acting on $(\mathrm{L^2}(\mathbb{R}, Bor(\mathbb{R}), \lambda)$ are well studied, we know that:
$$\text{For all } B \in Bor(\mathbb{R}) \text{ and } f \in \mathrm{L^2}(\mathbb{R}, Bor(\mathbb{R}), \lambda): E(B)f = \mathds{1}_{B} \,f$$
In particular, the action of the associated functional calculus can be derived as follows:

    $$\text{For all } \psi_1 \in \mathrm{D}(f(Q)) \text{ and } \psi_2 \in \mathrm{L^2}(\mathbb{R}, Bor(\mathbb{R}), \lambda):$$
\begin{align*}
    \inner{f(Q)\psi_1}{\psi_2}_{L^2} = \int\limits_{\sigma(Q)} f(x)d\mu_{\psi_1,\psi_2}(x) = \int\limits_{\mathbb{R}}f(x)\psi_1(x)\psi_2(x)d\lambda(x) = \inner{f\psi_1}{\psi_2}_{L^2}
\end{align*}
This shows that $f(Q)$ is nothing else than the multiplication operator $M_f$ on $\mathrm{L^2}(\mathbb{R}, Bor(\mathbb{R}), \lambda)$
By the Spectral Mapping Theorem for Regulated Functions:
$$\sigma(M_f) = \overline{f(\sigma_p(Q)) \cup \{f(c^+),f(c^-)\mid c \in \sigma(Q)\setminus\sigma_p(Q)\}}$$
However:
$$c \in \sigma_p(Q) \text{ if and only if } E(\{c\})\neq 0 \text{ if and only if } \lambda(\{c\}) >0$$
As singletons are Lebesgue-negligible, we get: $\sigma_p(Q) = \emptyset$. This proves that the spectrum of the multiplication operator for regulated functions reduces to the closure of one-sided limit values:
$$\sigma(M_f)=\overline{\{f(c^+), f(c^-)\mid c \in \mathbb{R}\}}$$
Oftentimes, quantum physicists are interested in the average value of the position operator $Q$, when acting on a quantum state $\psi$.
\begin{definition}
   Given the Hilbert space $(\mathrm{L^2}(\mathbb{R}, Bor(\mathbb{R}), \lambda), \inner{\cdot}{\cdot}_{L^2},\mathbb{R})$, and the position operator $(Q,D(Q))$, we define its expected value with respect to a quantum state \\ $\psi \in \mathrm{L^2}(\mathbb{R}, Bor(\mathbb{R}), \lambda)$ as follows:
   $$  \langle Q \rangle_{\psi}: = \int\limits_{\mathbb{R}}x (|\psi(x)|_{\mathbb{R}})^2\, d\lambda(x)$$
\end{definition}

With the help of our functional calculus, we can formulate the expected value spectrally:
$$\text{For all } \psi \in \mathrm{Q}:\inner{1(Q)\psi}{\psi}_{L^2} = \int\limits_{\sigma(Q)}x\,d\mu_{\psi,\psi}(x) = \int\limits_{\mathbb{R}}x  (|\psi(x)|_{\mathbb{R}})^2\, d\lambda(x) = \langle Q \rangle_{\psi}$$
It is worth mentioning that we cannot generally extend this equality to hold on all of $\mathrm{L^2}(\mathbb{R}, Bor(\mathbb{R}), \lambda)$. As justification why, we can consider the function $$\displaystyle  \psi(x) = \frac{1}{(1 + x^2)^{\frac{3}{4}}} \in \mathrm{L^2}(\mathbb{R}, Bor(\mathbb{R}),\lambda)\setminus \mathrm{D(Q)} $$

%% file: 8-App.tex
\chapter{An Application to Abstract Differential Equations} \label{ch: applications}
In this chapter, we will present Amnon Pazy's work on the Abstract Inhomogeneous Cauchy Problem, and see, with the help of an example, how we represent some solutions of this Initial Value Problem using the $HK$ functional calculus. With our chosen example, we will refine one of Pazy's results, enabling a simpler approximation of a certain type of solutions.

\medskip
Throughout this chapter, we will assume intermediary knowledge of $C_0$-semigroup theory. As a general rule of thumb, the reader should be familiar with the topics discussed in the first three chapters of Ioan Vrabie's book \enquote{$C_0$-Semigroups and Applications} \cite{Vrabie}.

\bigskip
We start our discussion by considering two functions $u,  f: [0,T] \longrightarrow \mathbb{R}$ together with a real-valued scalar $\lambda$ satisfying the following equation:
$$
\begin{cases}
\dot{u}(t) = \lambda u(t) + f(t)\\
u(0) = u_0
\end{cases}
$$
The solution to this ordinary differential equation is easily derived to be:
$$u(t) = e^{t\lambda} u_0 + \int\limits_0^t e^{(t -s)\lambda}f(s)\,ds \quad\text{ for all } t \in [0,T]$$
We may now generalise this Inhomogeneous Initial Value Problem by replacing the scalar $\lambda$ with a self-adjoint operator $(A,\mathrm{D}(A))$ generating a $C_0$-semigroup $\{T(t) \mid t \geq 0\}$ on some Banach space $(V, \|\cdot\|_V,\mathbb{F}).$
$$
\begin{cases}
\dot{u}(t) = A u(t) + f(t)\\
u(0) = u_0
\end{cases}
$$
It is clear from the formulation that $u, f: [0,T] \longrightarrow V$. From now on, we will refer to this problem as the Abstract Inhomogeneous Cauchy Problem, or $A.I.C.P.$ with datum $(u_0,f)$. To a large extent, the different types of solutions of this O.D.E. have been extensively studied by Amnon Pazy in his book \enquote{Semigroups of Linear Operators and Applications to Partial Differential Equations} \cite{Pazy}. Before turning this problem into a bona fide application of the $HK$ functional calculus,  we need to discuss Pazy's work in order to understand what types of solutions are admissible and under which conditions can they be considered. We will focus our study on two types of solutions: classical and mild. 
\begin{definition}[Classical Solution] Given a Banach space $(V, \|\cdot\|_V, \mathbb{F}),$ an infinitesimal generator $(A,\mathrm{D}(A))$ of a $C_0$-semigroup $\{T(t) \mid t\geq 0 \}$, a function $u: [0,T] \longrightarrow V$ is a \textbf{classical solution} of $A.I.C.P.$  with datum $(u_0,f)$, if the following three criteria are satisfied:
\begin{enumerate}
    \item $u \in \mathrm{C^1}((0,T); V) \cap \mathrm{C}([0,T];V)$
    \item For all $t \in (0,T): u(t) \in D(A)$
    \item $A.I.C.P.$  with datum $(u_0,f)$  is satisfied on $[0,T]$
\end{enumerate}
\end{definition}
The idea of a mild solution is motivated by the following derivation. For the $C_0$-semigroup $\{T(t)\mid t \geq 0\}$ we can construct a differentiable function on $(0,t)$.
$$g : [0,t] \longrightarrow V, \qquad s \longmapsto g(s): = T(t -s)u(s)$$
In particular:
\begin{align*}
    g'(s) 
    &= -T'(t -s)u(s) + T(t -s)u'(s)\\
    &= -AT(t-s)u(s) + T(t-s)[Au(s) +f(s)]\\
    &= -AT(t-s)u(s) + AT(t-s)u(s) + T(t-s)f(s)\\
    &= T(t-s) f(s)
\end{align*}
In the case $f \in \mathrm{L^1}(([0,T], Bor([0,t]), \lambda); V)$, we may wish to integrate our expression:
$$\int\limits_{0}^t g'(s)\,d\lambda(s) =  \int\limits_{0}^tT(t-s)f(s) \,d\lambda(s)$$
$$g(t) =g(0) + \int\limits_0^t T(t-s)f(s) \,d\lambda(s)$$
$$u(t) = T(t)u_0 + \int\limits_0^tT(t-s)f(s) \, d\lambda(s)$$
One important fact to keep in mind is that by the rather weak requirement of \\ $f \in \mathrm{L^1(}([0,T], Bor([0,t]), \lambda); V)$, if $f$ were to be a step function, $u$ will fail to be a classical solution as 
$$ \dot{u}(t) = Au(t) + f(t)$$
is discontinuous. It is therefore natural to formulate a broader type of solution.
\begin{definition}[Mild Solution]
  Given a Banach space $(V, \|\cdot\|_V, \mathbb{F}),$ an infinitesimal generator $(A,\mathrm{D}(A))$ of a $C_0$-semigroup $\{T(t) \mid t\geq 0 \}$, a function $u: [0,T] \longrightarrow V$ is a \textbf{weak solution} of $A.I.C.P.$  with datum $(u_0,f) \subseteq X \times \mathrm{L^1}(([0,T], Bor([0,t]), \lambda); V)$ if:
  \begin{enumerate}
      \item $u \in \mathrm{C}([0,T];V)$
      \item For all $t \in [0,T]: u(t) = T(t)u_0 + \displaystyle\int\limits_0^tT(t-s)f(s) \,d\lambda(s)$
  \end{enumerate}
\end{definition}
It is clear that for $f \in \mathrm{Reg([0,T];X)}$, we would only hope to get mild solutions. Furthermore, Pazy establishes an almost sufficient and necessary criterion for which mild solutions can become classical. 

It turns out that requiring $f \in \mathrm{C}([0,T];V)$ is still insufficient for this task.

For example, by fixing an infinitesimal generator $(A,\mathrm{D}(A))$ of a $C_0$-semigroup \\ $\{T(t) \mid t \geq 0\}$, $x \in V$ and  $T(t)x \notin D(A)$ for all $t \geq 0$, we can construct a continuous, integrable function:
$$ f : [0,T] \longrightarrow V, \qquad s \longmapsto f(s) : = T(s)x$$
With datum $(u_0,f)$ we form the $A.I.C.P.$:
$$
\begin{cases}
    \dot{u}(t) = Au(t) + T(t)x\\
    u(0) = 0
\end{cases}
$$
Its mild solution is: 
$$u(t)  = \int\limits_0^tT(t-s)T(s)x \, d\lambda(s) = tT(t)x \quad \text{for all }  t \in [0,T]$$
However, differentiability is only garantueed at $t = 0$,
\newpage
\begin{proposition}
$T(t)x$ is only differentiable at $t = 0$.
\end{proposition}
\begin{proof}
    $$T'(0): = \hspace{-0.2cm}\lim\limits_{\,\,\,h \,\to\, 0^+}\frac{hT(h)x}{h}  = x \,\, \text{ in } \,\,\|\cdot \|_V$$ 
    Let $h > 0$.
    \begin{align*}
        \lim\limits_{\,\,\,h \,\to\, 0^+}\frac{(t+h)T(t+h)x -tT(t)x}{h} 
        &=  \lim\limits_{\,\,\,h \,\to\, 0^+}\frac{tT(t+h)x - tT(t)x}{h} +  \lim\limits_{\,\,\,h \,\to\, 0^+}\frac{hT(t+h)x}{h}\\
        &= t\hspace{-0.2cm}\lim\limits_{\,\,\,h \,\to\, 0^+}\frac{T(t+h)x - T(t)x}{h} + T(t)x
    \end{align*}
As $T(t)x \notin \mathrm{D}(A), A(T(t)x): = \hspace{-0.2cm}\lim\limits_{\,\,\,h \,\to\, 0^+}\frac{T(t+h)x - T(t)x}{h}$ does not exist. \\
This proves our claim.
\end{proof}
\begin{theorem}[Pazy's Criterion for promoting Mild Solutions to Classical Ones] 
\mbox{}\\
Given a Banach space $(V, \|\cdot\|_V, \mathbb{F})$, an infinitesimal generator $(A,\mathrm{D}(A))$ of a $C_0$-semigroup $\{T(t) \mid t \geq 0\}$ and a function $f \in  \mathrm{C}([0,T];V)$, we consider the following map:
$$v : [0,T] \longrightarrow V, \qquad t \longmapsto\int\limits_0^tT(t-s)f(s) \, d\lambda(s)$$
If one of the following two properties is satisfied:
\begin{enumerate}
    \item The path $t \longmapsto v(t)$ is continuous on $[0,T]$.
    \item For all $t \in (0,T): v(t) \in \mathrm{D}(A)$ and $Av(t)$ is continuous.
 \end{enumerate}
\end{theorem}
Conversely, if for some $x \in \mathrm{D}(A)$, $A.I.C.P.$ admits a classical solution $u$ on $[0,T]$ with datum $(x,f)$, then both properties are satisfied.

\bigskip \bigskip
This theorem yields two corollaries:
\begin{corollary}
    Given a Banach space $(V, \|\cdot\|_V, \mathbb{F})$, an infinitesimal generator $(A,\mathrm{D}(A))$ of a $C_0$-semigroup $\{T(t) \mid t \geq 0\}$ and a continuously differentiable, $V$-valued function $f$ on $[0,T]$, for every datum $(x, f) \subseteq \mathrm{D}(A) \times \mathrm{C}([0,T];V)$ there exists a classical solution $u$.
\end{corollary}
\newpage
\begin{corollary}
     Given a Banach space $(V, \|\cdot\|_V, \mathbb{F})$, an infinitesimal generator $(A,\mathrm{D}(A))$ of a $C_0$-semigroup $\{T(t) \mid t \geq 0\}$ and a function $f \in \mathrm{C}((0,T);V)$ satisfying:
     $$ \text{For every } s \in (0,T): f(s) \in \mathrm{D}(A) \text { and } Af(s) \in \mathrm{L^1}(([0,T], Bor([0,t]), \lambda); V)$$
     then, for every datum $(x, f) \subseteq \mathrm{D}(A) \times \mathrm{C}([0,T];V)$, there exists a classical solution $u$.
\end{corollary}
Although the proofs of the theorem and its corollaries are not difficult, we will refer the reader to \textbf{Theorem 2.4} and its corollaries from Chapter IV on Pazy's aforementioned book \cite{Pazy}.

By using \textbf{Corollary 6.1}, Pazy managed to show that mild solutions can be uniformly approximated by classical solutions. 

\begin{theorem}[Mild Solutions are Uniform Limits of Classical Solutions]\label{thm:Pazyuniform}
     Given a Banach space $(V, \|\cdot\|_V, \mathbb{F})$, an infinitesimal generator $(A,\mathrm{D}(A))$ of a $C_0$-semigroup $\{T(t) \mid t \geq 0\}$, a function $f \in \mathrm{L^1}(([0,T], Bor([0,t]), \lambda); V)$, and a mild solution $u$ with datum $(x,f)$, there exists a sequence of strong solutions $\{u_n\}_{n \in \mathbb{N}}$ with data $(x_n,f_n)$ converging uniformly to $u$.
\end{theorem} 
\begin{proof}
    As $\{T(t) \mid t \geq 0\}$ is a $C_0$-semigroup, there exists $M \geq 1, \,\omega \in \mathbb{R}$ such that:
    $$\|T(t)\|_{op}\leq Me^{t\omega} \quad \text{ for all } t \geq 0$$
    Moreover, $(A,\mathrm{D}(A))$ generating a $C_0$-semigroup implies that $D(A)$ is dense in $(V, \|\cdot\|_V, \mathbb{F})$. 
    For the mild solution $u$ of $A.I.C.P.$ with datum $(x,f)$, since $\mathrm{C^1}([0,T];V)$ is dense in $\mathrm{L^1}(([0,T], Bor([0,t]), \lambda); V)$ with respect to $\|\cdot\|_{L^1}$, we have that our datum $(x,f)$ can be approximated by some sequence of data $\{(x_n,f_n)\}_{n \in \mathbb{N}}$. 
    
    Concretely:
    $$ x_n \xrightarrow[\, n \, \to \, \infty]{\,\|\cdot\|_{\mathbb{V}}}x \qquad \text{ and } \qquad f_n \xrightarrow[\, n \, \to \, \infty]{\,\|\cdot\|_{L^1}}f $$
    By \textbf{Corollary 6.1}, we have for every $n \in \mathbb{N}:$
    $$\begin{cases}
        \dot{u}_n(t) = Au_n(t) + f_n(t)\\
        u_n(0) = x_n
    \end{cases}$$
has a classical solution: 
$$u_n(t) = T(t)x_n + \int\limits_0^tT(t-s)f(s) \, d\lambda(s) \quad \text{ for every } t \in [0,T]$$
\newpage
By fixing some $t \in [0,T]$, we have:
\begin{align*}
    \|u_n(t) - u(t)\|_V 
    &= \left\|\,T(t)x_n + \int\limits_0^tT(t-s)f_n(s)\,d\lambda(s)  - T(t)x - \int\limits_0^tT(t-s)f_n(s)\,d\lambda(s) \,\right\|_V\\
    &\leq \left\|T(t)(x_n - x)\right\|_V + \int\limits_0^t\left\|T(t-s)(f_n(s)-f(s))\right\|_V\\
    &\leq Me^{t \omega} \|x_n - x\|_V + \int\limits_0^t Me^{\omega(t-s)}\|f_n(s) - f(s)\|_V \,d\lambda(s)\\
    &\leq Me^{|\omega|_{\mathbb{R}} T}\left(\|x_n - x\|_V + \|f_n(s) - f(s)\|_{L^1}\right)
\end{align*}
As the upper bound does not depend on the fixed $t$, we conclude:
$$u_n \xrightarrow[\, n \, \to \, \infty]{\,\|\cdot\|_{\infty}}u$$
This proves our claim.
\end{proof}
\bigskip

We can now examine a particular example.
\begin{example}
Fix some measure space $\left(\overline{(a,b)},Bor\left(\overline{(a,b)}\right), \lambda\right)$, where $(a,b)\subseteq\mathbb{R}$, and $a$ or $b$ may be infinite. Consider a bounded, non-decreasing sequence  $D =\{\lambda_n\}_{n \in \mathbb{N}}$ in $\overline{(a,b)}$, where each $\lambda_k$ satisfies $E(\{\lambda_k\}) = 0$. Based upon this sequence, we construct a function 
$g \in \mathrm{L^\infty}\left(\overline{(a,b)},Bor\left(\overline{(a,b)}\right), \lambda\right)$. Let $g_c \in \mathrm{C}_{\mathbb{R}}(\overline{(a,b)}\setminus D)$. By perturbing $g_c$ with discontinuities, we obtain:
$$
g(\lambda)=
\begin{cases}
    g_c(\lambda),  & \text{for }\,\, \lambda \in \overline{(a,b)}\setminus D\\
    \lambda_k,  & \text{for }\,\, \lambda = \lambda_k \in D
\end{cases}
$$
\end{example}
As the discontinuities of $g$ are countably many and are not essential ones, $g$ is regulated. We then consider the following $C_0$-semigroup $\{T(t) \mid t \geq 0\}$ characterised by its pointwise action:
$$\text{For all } f \in \mathrm{L^2}\left(\overline{(a,b)},Bor\left(\overline{(a,b)}\right), \lambda\right) \text{ and } t \geq 0: T(t)f(x) : = e^{t g(x)}f(x)$$
Before looking at the representation of the mild solution using the $HK$ functional calculus, we need to find out the infinitesimal generator $(A,\mathrm{D}(A))$ of $\{T(t) \mid t \geq 0\}$.
\newpage
Since $D$ is countably infinite, it is Lebesgue negligible. This is why we will restrict our analysis to $g_c$.

Let $h \in (0,\delta)$ for some fixed $\delta > 0$ and $x\in \overline{(a,b)}$. Then:
$$\lim\limits_{\,\,\,h \,\to \,0^+} \frac{T(h)f(x) - f(x)}{h} = \hspace{-0.2cm}\lim\limits_{\,\,\,h \to 0^+}\left(\frac{e^{hg_c(x)} - 1}{h}\right)f(x) = g(x)f(x) \quad \text{ in } \,\,\,\,\,\,|\cdot|_{\mathbb{R}}$$
This proves that:
$$\text{For all } x \in \overline{(a,b)}: \frac{T(h)f(x) - f(x)}{h} \xrightarrow[\,n \,\to\, \infty\,]{\,|\cdot|_{\mathbb{R}}\,} g(x)f(x) \quad \text{almost everywhere}$$
Moreover:
$$\left|\frac{e^{hg_c(x)} - 1}{h}\right|_{\mathbb{R}} = \left|\,\int\limits_0^1g_c(x)e^{shg_c(x)}\,d\lambda(s)\,\right|_{\mathbb{R}}\leq \|g\|_{\infty}e^{h \|g\|_{\infty}} \leq \|g\|_{\infty}e^{\delta \|g\|_{\infty}}$$
This enables us to use the dominated convergence theorem to conclude:
$$\left\|\frac{T(h)f - f}{h} - gf\right\|_{L^2} \xrightarrow[\,\,\,\, h \, \to \, 0^+\,]{\,|\cdot|_{\mathbb{R}}\,} \,\,\,0 $$
The infinitesimal generator $(A,\mathrm{D}(A))$ is hence globally defined, taking the form of a multiplication operator.
Concretely, $\mathrm{D}(A) = \mathrm{L^2}\left(\overline{(a,b)},Bor\left(\overline{(a,b)}\right), \lambda\right) $ and $A = M_g$. 

Since $g$ is real-valued, it follows that $A$ is self-adjoint. Moreover, as $A$ is closed, by \autoref{thm:closedgraphtheorem}, it is also bounded, implying that $\sigma(A)$ is compact.

\bigskip
For the sake of brevity, let $V: = \mathrm{L^2}\left(\overline{(a,b)},Bor\left(\overline{(a,b)}\right), \lambda\right)$.\\
Consider $$M: V \longrightarrow V, \quad f(x)\longmapsto xf(x) \quad \text{for all } x \in \overline{(a,b)}$$
We can now fix $x \in  V$ and $f \in  \mathrm{L^1}([0,T], Bor([0,T], \lambda);V)$. The mild solution $u$ of $A.I.C.P$ with datum $(x,f)$ is:
$$u(t) = T(t)x + \int\limits_{0}^tT(t-s) f(s) \, d\lambda(s) \quad \text{ for every } t \in [0,T]$$
\newpage
By the bounded $HK$ functional calculus, we may express $\{T(t)\mid t \geq 0\}$ as follows:
\begin{align*}
    T(t) 
    &= e^{tM_g}\\
    &= e^{tg(M)}\\
    &=\int\limits_{\sigma(M)}e^{tg(\lambda)}\, dE(\lambda)\\
    &= \int\limits_{\overline{(a,b)}\setminus D}e^{tg(\lambda)} \, dE(\lambda) + \int\limits_{D}e^{tg(\lambda)}\, dE(\lambda)\\
    &=  \int\limits_{\overline{(a,b)}\setminus D}e^{tg_c(\lambda)} \, dE(\lambda) + \sum\limits_{n=1}^\infty e^{t\lambda_k} \, dE(\{\lambda_k\})
\end{align*}

As $E(\{\lambda_k\}) = 0$ for all $k  \in \mathbb{N}$, we get:
$$T(t) = \int\limits_{\overline{(a,b})}e^{tg(\lambda)}\, dE(\lambda) = \int\limits_{\overline{(a,b)}\setminus D } e^{tg_c(\lambda)}\, dE(\lambda)$$
Our mild solution can be thus formulated as such:
$$u(t) = \bigg(\int\limits_{\overline{(a,b)} \setminus D}e^{tg_c(\lambda)} \, dE(\lambda)\bigg)x + \int\limits_0^t\bigg(\int\limits_{\overline{(a,b)} \setminus D}e^{(t-s)g_c(\lambda)} \, dE(\lambda)f(s)\bigg)\,d\lambda(s)$$
Our final intent with this example is to refine Pazy's result (\autoref{thm:Pazyuniform}) in the case $\overline{(a,b)} = [a,b]$, where $a,b <\infty$. As $[a,b]$ is compact, we approximate $\{T(t) \mid t \geq 0\}$ in the operator norm by simpler, bounded, linear operators $\{T_n(t) \mid t \geq 0\}$. In our case, as $\{e^{tg}\ \mid t \geq 0\} \subseteq \mathrm{Reg}_{\mathbb{F}}([a,b])$, we can construct a sequence of step functions:
$$\text{For all } n \in \mathbb{N} \text{ and } t\geq 0: s_{t,n} = \sum_{k=1}^{M_n}e^{tg(\eta_{k,n})}\,\,\mathds{1}_{I_{k,n}}$$
such that:
\begin{enumerate}
    \item For all $t \geq 0: \|s_{t,n} - e^{tg}\|_{\infty} \xrightarrow[\, n \, \to \, \infty]{\,|\cdot|_{\mathbb{R}}} 0$
    \item For all $n \in \mathbb{N}$ and $t \geq 0: \|s_{t+h,n} - s_{t,n}\|_{\infty }\xrightarrow[\, h \, \to \, 0]{\,|\cdot|_{\mathbb{R}}} 0$
\end{enumerate}
where:
$$\widehat{P}_{\gamma,n}: = \left\{(\eta_{k,n}, I_{k,n}, \gamma) \mid k=\overline{1,M_n}\right\}$$
defines a sequence (in $n \in \mathbb{N}$) of gauge-fine, tagged partitions of $[a,b]$, in which $\gamma$ is the same gauge function defined in \autoref{ch:henstockkurzweilbounded} for step functions.

Our candidate for the approximant of the $C_0$-semigroup $\{T(t) \mid t \geq 0\}$ is:
$$T_n(t): = \int\limits_{[a,b]}s_{t,n}(\lambda)\, dE(\lambda) = \sum\limits_{k=1}^{M_n}e^{tg(\eta_{k,n})}E(I_{k,n}) \quad \text{ for all } t \geq 0$$
By \autoref{lem:lipschitzcontinuity}, we will get two facts:
\begin{enumerate}
    \item For each $t \geq 0: T_n(t) \xrightarrow[\, n \, \to \, \infty]{\,\|\cdot\|_{L^2}}T(t)$
    \item For each $t \geq 0: t \longmapsto T_n(t)$ is continuous in the operator norm
\end{enumerate}
Two independent facts that we get as a consequence of how we have constructed our step functions are that:
\begin{enumerate}
    \item For all $ n \in \mathbb{N}: T_n(0) = I $
    \item For all $n \in \mathbb{N}$ and $t,s \geq 0: T_n(t+s) = T_n(t)T_n(s)$
\end{enumerate}

This shows that $\{T_n(t) \mid t \geq 0\}_{n \in \mathbb{N}}$ is a sequence of $C_0$-semigroups. In particular, for the same datum $(x,f) \subseteq V \times \mathrm{L^1}(([0,T], Bor([0,T]),\lambda);V)$ as the mild solution $u$ of $A.I.C.P.$, we can construct a new sequence of mild solutions on $[0,T]$:
\begin{align*}
    w_n(t)
    &:= T_n(t)x + \int\limits_0^tT_n(t-s)f(s)\, d\lambda(s)\\
    &\hspace{0.1cm}=\bigg(\int\limits_{[a,b]\setminus D}s_{t,n}(\lambda)\, dE(\lambda)\bigg)x + \int\limits_0^t\bigg(\int\limits_{[a,b]\setminus D}s_{t-s,n}(\lambda)\, dE(\lambda)\bigg)f(s)\,d\lambda(s)
    \end{align*}
By using Pazy's result (\autoref{thm:Pazyuniform}), each $w_n$ can be uniformly approximated on $[0,T]$ by a sequence of classical solutions $\{w_{m,n}\}_{m \in \mathbb{N}}$. Moreover, $u$ can be seen as the uniform limit of $\{w_n\}_{n\in\mathbb{N}}$.
\begin{align*}
    \hspace{-0.6cm}\sup\limits_{t \in [0,T]}\|w_n(t) - u(t)\|_V 
    &\leq \sup\limits_{t \in [0,T]}\left\|\,T_n(t)x + \int\limits_0^tT_n(t-s)f(s) \, d\lambda(s)  - T(t)x - \int\limits_0^tT(t-s) f(s) \, d\lambda(s)\,\right\|_V\\
    &\leq \sup\limits_{t \in [0,T]}\left\|T_n(t) - T(t)\right\|_{op} \,\|x\|_V + \sup\limits_{t \in [0,T]}\left\|T_n(t) - T(t)\right\|_{op} \int\limits_0^t\|f(s)\|_V \, d\lambda(s)\\
    &\leq \sup\limits_{t \in [0,T]}\left\|T_n(t) - T(t)\right\|_{op}\left(\|x\|_V + \|f\|_L^1\right)
\end{align*}
As we let $n \to \infty$, we get:
$$w_n \xrightarrow[\, n \, \to \, \infty]{\,\|\cdot\|_{\infty}} u$$

\newpage
Pazy's result applied to $u$ gives us our standard uniform convergent sequence $\{u_l\}_{l \in \mathbb{N}}$.
By two direct triangle inequalities, we have:
\begin{enumerate}
    \item $\sup\limits_{t \in [0,T]}\|u(t) - w_{m,n}(t)\|_V \leq \sup\limits_{t \in [0,T]}\|u(t) - u_n(t)\|_V + \sup\limits_{t \in [0,T]}\|u_n(t) - w_{m,n}(t)\|_V \xrightarrow[\,n \,\to\, \infty\,]{\,|\cdot|_{\mathbb{R}}\,} 0$
    \item $\sup\limits_{t \in [0,T]}\|u_n(t) - w_{m,n}(t)\|_V \leq \sup\limits_{t \in [0,T]}\|u_n(t) - u(t)\|_V + \sup\limits_{t \in [0,T]}\|u(t) - w_{m,n}(t)\|_V \xrightarrow[\,l,m,n \,\to\, \infty\,]{\,|\cdot|_{\mathbb{R}}\,}\nobreak 0$
\end{enumerate}
The  first result shows that $u$ can be described as the uniform limit of $\{u_{m,n}\}_{(m,n) \subseteq \mathbb{N}\times \mathbb{N}}$.
The second result underlines the fact that the two classical uniform approximants of the mild solution $u$ are asymptotically close together. This is of great advantage, as it gives the user a wider choice of classical approximants. 

%% file: 9-Conclusions.tex
\chapter{Conclusions and Further Directions}
\section*{Conclusions}
In our chapters, we have:
\begin{itemize}[label= -]
    \item used Henstock-Kurzweil Integration techniques to construct a functional calculus differently.
    \item proved the Spectral Mapping Theorem by means of our functional calculus.
    \item studied the advantages of the regulated functions over larger spaces such as the space of bounded, Borel-measurable functions.
    \item given an example of how one represents solutions of Abstract Differential Equations through this functional calculus.
\end{itemize}
\section*{Further Directions}
\begin{itemize}[label = -]
    \item given a Banach Space $(V, \|\cdot\|_V, \mathbb{F})$, a normal operator $(A,\mathrm{D}(A))$, if possible, on what function space can we construct an $HK$ functional calculus, and how.
\end{itemize}

    The reason why we have worked with self-adjoint operators is that the most primitive notion of the Henstock-Kurzweil integration techniques stemmed from generalising the Riemann Integral on the real line. Since normal operators may have spectra lying in the complex plane, one needs to study $HK$ integration in such spaces first. 
    \newpage 
    One approach is to start off with a contour integral 
    $\displaystyle\int\limits_{\Omega}f(z)\,dz$
    for some function 
    \\ $f: \Omega \subseteq \mathbb{C} \longrightarrow \mathbb{C}$. By considering a path $\gamma \in C_{\mathbb{C}}^1([a,b])$, we can define: 
    $$F: [a,b] \longrightarrow \mathbb{C}, \quad t \longmapsto F(t):= f(\gamma(t))\gamma'(t)$$
    Breaking $F$ into its real and imaginary parts $F= \Re(F) + i\Im(F)$, we may demand the following definition of Henstock-Kurzweil Integrability of f.
    \begin{definition}[$\gamma-HK$ integrability]
    Given a function $f: \Omega \subseteq \mathbb{C} \longrightarrow\mathbb{C}$ and a path $\gamma \in C_{\mathbb{C}}^1([a,b])$ such that $\gamma([a,b]) \subseteq \Omega$, we say that $f$ is  \textbf{$\gamma-HK$ integrable} if $\Re((f\circ \gamma)\cdot\gamma')$ and $\Im((f \circ \gamma)\cdot\gamma')$ are $HK$ integrable in the usual sense.
    \end{definition}
    We may use the same reasoning as in the  $HK$ functional calculus, but tweak the integration methods by this new definition.
    We encourage the reader to adjust this definition, as long as it extends to a fruitful research area, and to embark on developing the theory further.

%% file: 10-Appendix.tex
\appendix